\documentclass[12pt]{amsart}
\usepackage{amsfonts, amsmath, amssymb, amsthm, enumerate, float, graphicx, hhline, url}
\usepackage[all,knot,poly]{xy}

\makeatletter
\@namedef{subjclassname@2020}{\textup{2020} Mathematics Subject Classification}
\makeatother

\newtheorem{theorem}{Theorem}[section]
\newtheorem{lemma}[theorem]{Lemma}
\newtheorem{proposition}[theorem]{Proposition}
\newtheorem{conjecture}[theorem]{Conjecture}

\theoremstyle{definition}
\newtheorem{definition}[theorem]{Definition}
\newtheorem{remark}[theorem]{Remark}

\numberwithin{equation}{section}

\begin{document}
	
\baselineskip=17pt
	
\title[Hecke algebra trace algorithm]{Hecke algebra trace algorithm and some conjectures on weaving knots}

\author[R. Mishra]{Rama Mishra}
\address{Department of Mathematics\\ Indian Institute of Science Education and Research\\ Dr. Homi Bhabha Road\\ Pashan\\ Pune 411008\\ India}
\email{r.mishra@iiserpune.ac.in}

\author[H. Raundal]{Hitesh Raundal}
\address{Bhaskaracharya Pratishthana\\ 56/14\\ Damle Path\\ Off Law College Road\\ Erandavane\\ Pune 411004\\ India}
\email{hiteshrndl@gmail.com}

\date{}

\begin{abstract}
Computing polynomial invariants for knots and links using braid representations relies heavily on finding the trace of Hecke algebra elements. There is no easy method known for computing the trace and hence it becomes difficult to compute the known polynomial invariants of knots using their braid representations. In this paper, we provide an algorithm to compute the trace of the Hecke algebra representation of any braid. We simplify this algorithm and write a Mathematica program to compute the invariants such as Alexander polynomial, Jones polynomial, HOMFLY-PT polynomial and Khovanov homology of a very special family of knots and links $W(n,m)$ known as weaving knots by expressing them as closure of weaving braids. We also explore on the relationship between the topological and geometric invariants of this family of alternating and hyperbolic knots (links) by generating data for the subfamilies $W(3,m)$, $W(4,m)$, $W(5,m)$ and $W(6,m)$ of weaving knots.
\end{abstract}
	
	
\maketitle

\section{Introduction}\label{sec1}

Hecke algebras are quotients of group rings of Artin's braid groups \cite{b,bz} and play important role in the study of many branches of mathematics. One major application of Hecke algebras is in defining link invariants (for example, see \cite{hkw,j2}). V. F. R. Jones \cite{j2} used a representation of braid groups into Hecke algebras which he combined with a special trace function on Hecke algebras and showed the existence of a two variable polynomial as a link invariant. This two variable polynomial was first discovered by Freyd and et al \cite{homfly,pt} and it is a universal skein invariant in the sense that all known polynomial invariants could be derived from this only by change of variables. Thus, if one is aiming at computing various polynomial invariants for a given knot or for a given family of knots, it is a good idea to compute this two variable polynomial. Assuming that one knows a nice braid representation for the knot under consideration, the most important ingredient in this computation will be finding the trace of the Hecke algebra element which represents the braid. The trace function on Hecke algebras is defined by certain axioms and can be calculated in several steps. As a general rule, what one would do is that write an element in a Hecke algebra as a linear combination of elements in a suitable basis and find the trace of the basis elements and use the linearity property of the trace function . However, there is no effective way of doing this in general is known.

In \cite{ms}, the authors had worked out a recursive method of computing the trace for the Hecke algebra representation of special $3$-braids whose closure gives an infinite family of hyperbolic knots and links known as weaving knots of type $(3,m)$, denoted by $W(3,m)$. In general, {\it a weaving knot of type $(n,m)$, denoted by $W(n,m)$}, is the knot or link that is the closure of the braid $\left(\sigma_1\sigma_2^{-1}\sigma_3\sigma_4^{-1}\cdots\sigma_{n-1}^\delta\right)^m$, where $\delta=1$ if $n$ is even and $\delta=-1$ if $n$ is odd. For example, the weaving knot $W(4,5)$ is shown in Figure \ref{fig1}. By the definition itself, weaving knots can include links with many components, and throughout this paper, {\it weaving knots will denote both knots and links.} It was simple to compute the trace for the Hecke algebra representation of $3$-braids whose closures are weaving knots $W(3,m)$, since the braid group on $3$-strands is represented in the Hecke algebra $H_3(q)$ which has dimension $3!$ and thus expressing an element as a linear combination of basis elements was easy and could be done by hand. The same technique does not work out in case of weaving knots $W(n,m)$ for $n\geq4$. In this paper, we give an algorithm to express the Hecke algebra representation of a weaving $n$-braid (the closure is a weaving knot $W(n,m)$ for some $m$) as a linear combination of elements in a suitable basis for the Hecke algebra $H_n(q)$. We also write down an algorithm to compute the trace of basis elements of $H_n(q)$ for $n\geq 3$. Combining these two algorithms we obtain an algorithm to compute the trace of the representation of a weaving $n$-braid in $H_n(q)$. We use this algorithm for the trace to compute the two variable polynomial invariant for weaving knots and hence derive the Alexander polynomial \cite{a}, the Jones polynomial \cite{j1} and the HOMFLY-PT polynomial \cite{homfly,pt} for weaving knots. Since weaving knots $W(n,m)$ are alternating knots for $m$ and $n$ coprime, we utilize the results of Lee \cite{l1,l2} and Shumakovitch \cite{s1,s2} to compute the Khovanov homology \cite{k} of these knots. We normalize the Khovanov ranks (the ranks of the Khovanov homology groups) along a line and provide the evidence that the normalized ranks approach to a normal distribution. In some sense this paper extends the results for $W(3,m)$ in \cite{ms} to the general case $W(n,m)$ for all $n$ and $m$.

On the geometry side, William Thurston's seminal result \cite[Corollary 2.5]{t} in the 1980s along with Mostow's rigidity theorem \cite[Theorem 3.1]{t} ensure that most knot complements have a unique structure of a hyperbolic manifold. These two results together establishes a strong connection between hyperbolic geometry and knot theory, since knots are determined by their complements. Many mathematicians are naturally interested in finding out if any inference about some of the geometric invariants (such as hyperbolic volume of the complement) of a knot can be derived from any of its topological invariants (such as Jones polynomial or colored Jones polynomial). The ultimate curiosity in this direction is the validity of the open conjecture known as {\it volume conjecture} \cite{m}. The weaving knots $W(n,m)$ are good candidates for such exploration. Champanerkar, Kofman and Purcell \cite{ckp2} provided asymptotically sharp bounds for the relative volume (i.e. volume divided by crossing number) of these knots (links). They showed that
\begin{equation}
\lim_{n,m\to\infty}\frac{vol\left(S^3\setminus W(n,m)\right)}{c\left(W(n,m)\right)}=v_{oct}
\end{equation}
\noindent and hence according to their work in \cite{ckp1} they conclude that weaving knots are geometrically maximal. Since a weaving knot $W(n,m)$ with $\gcd(n,m)=1$ is an alternating knot, once we have found the Jones polynomial of $W(n,m)$ we can find a bound in terms of the twist number $T(W(n,m))$ of $W(n,m)$ using a result of Dasbach and Lin \cite{dl}. The twist number $T(L)$ of a knot $L$ is nothing but the sum of the modulus of coefficients of $t^{l+1}$ and $t^{h-1}$ in the Jones polynomial of $L$, where $l$ and $h$ are the lowest and the highest degrees of the Jones polynomial of $L$. In the same paper, Dasbach and Lin defined invariants such as $i^\text{th}$ twist number $T_i(L)$ of a knot $L$ to be the sum of the modulus of coefficients of $t^{l+i}$ and $t^{h-i}$ in the Jones polynomial of $L$. With the data provided in their paper, they observed that the twist numbers correlate with the hyperbolic volume of the knot complement. Since we can write the Jones polynomial of a weaving knot $W(n,m)$, we can compute the twist numbers $T_i\left(W(n,m)\right)$ for $i$ within the span of the Jones polynomial, that is, for $1\leq i<\frac{mn}{2}$. We performed experiments on the values of twist numbers $T_i(W(n,m))$ of weaving knots $W(n,m)$ and came up with four set of bounds on the relative volume of these knots (links) that seem to provide better bounds than given in \cite{ckp2}.

The paper is organized as follows: In Section \ref{sec2}, we include the basics on Hecke algebras, its trace and some relevant results. In Section \ref{sec3}, we prove some lemmas and propositions that describe our algorithm to compute the trace of the representation of any braid in a Hecke algebra. Furthermore, we prove a theorem that gives a simplified algorithm for the trace in case of weaving braids. At the end of the section, we include the Mathematica program to compute the trace of the Hecke algebra representation of weaving braids. We also set up notations that are used in describing the two variable polynomial invariant (i.e. the universal skein invariant). In Section \ref{sec4}, we explain how to derive the universal skein invariant for weaving knots from the trace (of the Hecke algebra representation of weaving braids) that is already computable using the algorithm in Section \ref{sec3}. We use the appropriate substitutions to write down the polynomial invariants such as the Alexander polynomial, the Jones polynomial and the HOMFLY-PT polynomial. We also include the additional Mathematica programs to compute these polynomial invariants and add lists of examples of the polynomial invariants for some weaving knots. In Section \ref{sec5}, we discuss the higher twist numbers introduced by Dasbach and Lin \cite{dl}. We provide important observations about higher twist numbers of weaving knots and show that the bounds for the relative volume can be improved from the bounds given in \cite{ckp2}. Section \ref{sec6} discusses a computation of the Khovanov homology of weaving knots $W(n,m)$ for $n$ and $m$ coprime. In this section, we also discuss the distribution of the normalized Khovanov ranks (the normalized ranks of the Khovanov homology groups) along a line. At the end we provide the tables displaying the evidence of our observations that we are going to conjecture in this paper.

\section{Hecke algebras: some observations}\label{sec2}

We review briefly the definition of a Hecke algebra $H_{i+1}(q)$ \cite{b,hkw} over a field $K$ and discuss a suitable basis for $H_{i+1}(q)$ as a vector space over $K$. We prove some lemmas and propositions that are useful in proving our main results.

\begin{definition}
Working over a ground field $K$ containing an element $q\neq0$, the Hecke algebra $H_{i+1}(q)$ is an associative algebra with unity, generated by $T_1,T_2,\ldots,T_i$ and satisfying the following relations:
\begin{enumerate}[(1)]
\item $T_jT_k=T_kT_j$\quad whenever $\left|j-k\right|\geq2$
\item $T_jT_{j+1}T_j =T_{j+1}T_jT_{j+1}$\quad for $1\leq j\leq i-1$, and
\item ${T_j}^2 = (q-1)T_j+q$\quad for $j=1,2,\ldots,i$.
\end{enumerate}
\end{definition}

For $j=1,2,\ldots,i$, recasting the relation ${T_j}^2 = (q-1)T_j+q$ in the form $q^{-1}\left(T_j+(1-q)\right)T_j=1$ shows that $T_j$ is invertible in $H_{i+1}(q)$ with ${T_j}^{-1}=q^{-1}\left(T_j+(1-q)\right)$.

For a positive integer $k$, define $P_k(q)=(-1)^{k-1}\sum_{j=0}^{k-1}(-q)^j$, $P_{-k}(q)=q^{-k}\sum_{j=0}^{k-1}(-q)^j$ and $P_0(q)=0$. For $k\in\mathbb{Z}$, one can check the following:
\begin{enumerate}[(1)]
\item $qP_{k-1}(q)+(q-1)P_k(q)=P_{k+1}(q)$,
\item $P_k(q)=\dfrac{q^k-(-1)^k}{q+1}$ if $q\neq-1$, and
\item $P_k(q)=(-1)^{k-1}k$ if $q=-1$.
\end{enumerate}

\begin{proposition}\label{prop1}
For $j=1,2,\ldots,i$, we have the following relation in $H_{i+1}(q)$:
\begin{equation*}
{T_j}^k=P_k(q)T_j+qP_{k-1}(q)\quad\text{for all}\;\,k\in\mathbb{Z}\,.
\end{equation*}
\end{proposition}

\begin{proof}
Note that ${T_j}^0=P_0(q)T_j+qP_{-1}(q)$, since ${T_j}^0=1$, $P_0(q)=0$ and $P_{-1}(q)=q^{-1}$. We now prove the proposition for $k\neq 0$. Let us consider the following two cases: 
\begin{enumerate}[(1)]
\item To prove the proposition for $k\in\mathbb{Z}^+$: We prove the equality ${T_j}^k=P_k(q)T_j+qP_{k-1}(q)$ by induction on $k$. Note that ${T_j}^1=P_1(q)T_j+qP_0(q)$, since $P_1(q)=1$ and $P_0(q)=0$. Assume that ${T_j}^k=P_k(q)T_j+qP_{k-1}(q)$ for $k\geq1$. We show that ${T_j}^{k+1}=P_{k+1}(q)T_j+qP_k(q)$. Using the induction hypothesis, we have the following estimate:
\begingroup
\allowdisplaybreaks
\begin{align*}
{T_j}^{k+1}&={T_j}^kT_j\\
&=\left(P_k(q)T_j+qP_{k-1}(q)\right)T_j\\
&=P_k(q){T_j}^2+qP_{k-1}(q)T_j\\
&=P_k(q)\left((q-1)T_j+q\right)+qP_{k-1}(q)T_j\\
&=\left(qP_{k-1}(q)+(q-1)P_k(q)\right)T_j+qP_k(q)\\
&=P_{k+1}(q)T_j+qP_k(q)\,,
\end{align*}
\endgroup
\noindent since ${T_j}^2=(q-1)T_j+q$ and $qP_{k-1}(q)+(q-1)P_k(q)=P_{k+1}(q)$.
\item To prove the proposition for $k\in\mathbb{Z}^-$: We need to prove that ${T_j}^{-l}=P_{-l}(q)T_j+qP_{-l-1}(q)$ for $l\in\mathbb{Z}^+$. We prove this equality by induction on $l$. Note that ${T_j}^{-1}=P_{-1}(q)T_j+qP_{-2}(q)$, since ${T_j}^{-1}=q^{-1}\left(T_j+(1-q)\right)$, $P_{-1}(q)=q^{-1}$ and $P_{-2}(q)=q^{-2}(1-q)$. Assume that ${T_j}^{-l}=P_{-l}(q)T_j+qP_{-l-1}(q)$ for $l\geq1$. We show that ${T_j}^{-l-1}=P_{-l-1}(q)T_j+qP_{-l-2}(q)$. Using the induction hypothesis, we have the following:
\begingroup
\allowdisplaybreaks
\begin{align*}
{T_j}^{-l-1}&={T_j}^{-l}{T_j}^{-1}\\
&=\left(P_{-l}(q)T_j+qP_{-l-1}(q)\right){T_j}^{-1}\\
&=P_{-l}(q)+qP_{-l-1}(q)q^{-1}\left(T_j+(1-q)\right)\\
&=P_{-l-1}(q)T_j+P_{-l}(q)+(1-q)P_{-l-1}(q)\\
&=P_{-l-1}(q)T_j+qP_{-l-2}(q)\,,
\end{align*}
\endgroup
\noindent since ${T_j}^{-1}=q^{-1}\left(T_j+(1-q)\right)$ and $P_{-l}(q)+(1-q)P_{-l-1}(q)=qP_{-l-2}(q)$.\qedhere
\end{enumerate}
\end{proof}

We quote the following result (the proof can be seen in \cite{hkw}).
\begin{theorem}
Let $z$ be an element in the field $K$. There exists a unique family of $K$-linear trace functions $Tr:H_i(q)\to K$ compatible with the inclusions $H_i(q)\hookrightarrow H_{i+1}(q)$ and satisfying the following properties:
\begin{enumerate}[(1)]
\item $Tr(1)=1$,
\item $Tr(ab)=Tr(ba)$, and
\item $Tr (aT_ib)=z\,Tr(ab)$
\end{enumerate}
for $a,b\in H_i(q)$ and $i\geq1$.
\end{theorem}

Using the relations in $H_{i+1}(q)$ and the properties of the trace function, we can observe that the trace of an element is a Laurent polynomial in $q$ and $z$. In principle, the properties of the trace function help us to compute the trace of any element in $H_{i+1}(q)$. However, the elements in $H_{i+1}(q)$ can be very complicated. So, the first step should be to express an element as a linear combination of elements in some basis of $H_{i+1}(q)$. Thus, we must throw some light on a nice basis for $H_{i+1}(q)$.

Let us consider the sets as follows:
\begingroup
\allowdisplaybreaks
\begin{align*}
U_1&=\left\{1,\,T_1\right\},\\[5pt]
U_2&=\left\{1,\,T_2,\,T_2T_1\right\},\\
\vdots&\\[5pt]
U_i&=\left\{1,\,T_i,\,T_iT_{i-1},\,T_iT_{i-1}T_{i-2},\,\ldots,\,T_iT_{i-1}\cdots T_1\right\}.
\end{align*}
\endgroup
\noindent The Hecke algebra $H_{i+1}(q)$ (as a vector space over $K$) has a basis
\begin{equation*}
\mathcal{B}_i=\left\{u_1u_2\cdots u_i\in H_{i+1}(q)\mid u_j\in U_j\;\text{for}\;j=1,2,\ldots, i\right\}\,.
\end{equation*}
\noindent This is well known in the theory of Hecke algebras and an elegant proof of this fact can be found in \cite{hkw}. Let $\mathcal{C}_i$ be the subcollection of $\mathcal{B}_i$ defined as follows:
\begin{equation*}
\mathcal{C}_i=\left\{u_1u_2\cdots u_i\in\mathcal{B}_i\mid u_j\in\left\{1,T_j\right\}\;\text{for}\;j=1,2,\ldots, i\right\}\,.
\end{equation*}
\noindent By convention, let $U_0=\mathcal{B}_0=\mathcal{C}_0=\{\,1\,\}$. For $i\geq0$, consider the set $U_i$. In this paper, we use the convention that $1$ is the zeroth element of $U_i$, $T_i$ is the first element, $T_iT_{i-1}$ is the second element, $T_iT_{i-1}T_{i-2}$ is the third element and so on up to $T_iT_{i-1}\cdots T_1$ which is the $i^\text{th}$ element of $U_i$. For $j=0,1,\ldots,i$, let $u_i^j$ denote the $j^{\text{th}}$ element of $U_i$, i.e.
\begin{equation}
u_i^j=
\left\{
\begin{array}{ll}
1 & \mbox{if}\;j=0\,,\\[5pt]
T_iT_{i-1}\cdots T_{i-j+1} & \mbox{if}\;j\neq0\,.
\end{array}
\right.
\end{equation}
\noindent Note that the length of $u_i^j$ (as a word in $T_k$'s) is $j$. By multiplying $T_{i+1}$ on the left of $u_i^j$ gives an element of $U_{i+1}$ of length $j+1$. That is, $u_{i+1}^{j+1}=T_{i+1}u_i^j$ for $0\leq j\leq i$ and $i\geq0$.

Consider the following sets:
\begingroup
\allowdisplaybreaks
\begin{align*}
\mathcal{L}_i&=\left\{\left(l_1,l_2,\ldots,l_i\right)\in\mathbb{Z}^i\mid l_j\in\left\{0,1,,\dots,j\right\}\;\text{for}\;j=1,2,\ldots,i\right\},\\[5pt]
\mathcal{M}_i&=\left\{\left(l_1,l_2,\ldots,l_i\right)\in\mathcal{L}_i\mid l_j\;\text{is either}\;0\;\text{or}\;1\;\text{for}\;j=1,2,\ldots,i\right\}\quad\text{and}\\[5pt]
\mathcal{N}_i&=\left\{\left(l_1,l_2,\ldots,l_i\right)\in\mathcal{M}_i\mid l_j=1\;\text{if}\;j\;\text{is odd}\right\}.
\end{align*}
\endgroup
\noindent For $l=\left(l_1,l_2,\ldots,l_i\right)$ in $\mathcal{L}_i$, let $\beta_i^l$ denote the element $u_1^{l_1}u_2^{l_2}\cdots u_i^{l_i}$ in the basis $\mathcal{B}_i$ of $H_{i+1}(q)$. In other words, $\beta_i^l$ is the product of the ${l_1}^{\!\text{th}}$ element of $U_1$, the ${l_2}^{\!\text{th}}$ element of $U_2$ and so on up to the ${l_i}^{\!\text{th}}$ element of $U_i$. Note that the length of $\beta_i^l$ (as a word in $T_k$'s) is $l_1+l_2+\cdots+l_i$. The basis $\mathcal{B}_i$ for $H_{i+1}(q)$ and its subcollection $\mathcal{C}_i$ can be written as follows:
\begin{equation*}
\mathcal{B}_i=\left\{\beta_i^l\mid l\in\mathcal{L}_i\right\}\qquad\text{and}\qquad\mathcal{C}_i=\left\{\beta_i^m\mid m\in\mathcal{M}_i\right\}.
\end{equation*}
\noindent By convention, let $\mathcal{L}_0=\mathcal{M}_0=\mathcal{N}_0=\left\{\{\;\}\right\}$ and let $\beta_0^l=1$ for $l=\{\:\}$. For $l\in\mathcal{L}_i$ and $0\leq j\leq i+1$, by multiplying $u_{i+1}^j$ on the right of $\beta_i^l$ gives the element $\beta_{i+1}^{l^\prime}\in\mathcal{B}_{i+1}$ for $l^\prime=(l,j)$ in $\mathcal{L}_{i+1}$. Similarly, for $m\in\mathcal{M}_i$ and $0\leq j\leq1$, the product $\beta_i^mu_{i+1}^j$ is the element $\beta_{i+1}^{m^\prime}\in\mathcal{C}_{i+1}$ for $m^\prime=(m,j)$ in $\mathcal{M}_{i+1}$.

\section{Algorithm to find the trace}\label{sec3}

The specification $\rho(\sigma_i) =T_i$, for $i=1,2,\ldots,n$, defines a representation of the group $B_{n+1}$ of braids on $n+1$ strands into the multiplicative monoid of $H_{n+1}(q)$ (see \cite[Section 6]{hkw}), where $\sigma_i$'s are the elementary braids generating $B_{n+1}$. In this section, we provide an algorithm to compute the trace of an element $\rho(\alpha)$ in $H_{n+1}(q)$ for $\alpha\in B_{n+1}$. This requires several steps. We present step by step algorithms to arrive at a final algorithm for computing the trace. We specialize the algorithm to compute the trace of the representation of weaving braids, in which case, it gets simplified.

\begin{lemma}\label{lem3}
For a positive integer $i$, we have the following relations in $H_{i+1}(q)$:
\begin{enumerate}[(1)]
\item $u_i^su_i^{s^\prime}=(q-1)u_{i-1}^{s-1}u_i^{s^\prime}+qu_{i-1}^{s^\prime-1}u_i^{s-1}$\quad if $1\leq s\leq s^\prime\leq i$,\quad and
\item $u_i^su_i^{s^\prime}=u_{i-1}^{s^\prime}u_i^s$\quad if $1\leq s^\prime<s\leq i$.
\end{enumerate}
\end{lemma}

\begin{proof}
We prove the lemma by considering the parts (1) and (2) separately.
\begin{enumerate}[(1)]
\item To prove the first equality: We prove the equality by induction on $i$. For $i=1$, the condition $1\leq s\leq s^\prime\leq i$ becomes $s=s^\prime=1$; thus, by the third relation in $H_2(q)$,
\begin{equation*}
u_1^su_1^{s^\prime}={T_1}^2=(q-1)T_1+q=(q-1)u_0^{s-1}u_1^{s^\prime}+qu_0^{s^\prime-1}u_1^{s-1}.
\end{equation*} 
\noindent The equality holds for $i=1$. Assume the equality for a positive integer $i-1$. We prove the equality for the next integer. Consider a product $u_i^su_i^{s^\prime}$ for $1\leq s\leq s^\prime\leq i$. Depending on $s=1$ or $s\geq2$, we have the following cases:
\begin{enumerate}[(a)]
\item If $s=1$: By the third relation in $H_{i+1}(q)$,
\begingroup
\allowdisplaybreaks
\begin{align*} 
\hskip13mm u_i^su_i^{s^\prime}&=T_iT_iu_{i-1}^{s^\prime-1}=(q-1)T_iu_{i-1}^{s^\prime-1}+qu_{i-1}^{s^\prime-1}=(q-1)u_i^{s^\prime}+qu_{i-1}^{s^\prime-1}\\
&=(q-1)u_{i-1}^{s-1}u_i^{s^\prime}+qu_{i-1}^{s^\prime-1}u_i^{s-1}.
\end{align*}
\endgroup
\item If $s\geq2$: By the first and second relations in $H_{i+1}(q)$,
\begingroup
\allowdisplaybreaks
\begin{align*}
u_i^su_i^{s^\prime}&=T_iT_{i-1}u_{i-2}^{s-2}T_iu_{i-1}^{s^\prime-1}=T_iT_{i-1}T_iu_{i-2}^{s-2}u_{i-1}^{s^\prime-1}\\
&=T_{i-1}T_iT_{i-1}u_{i-2}^{s-2}u_{i-1}^{s^\prime-1}=T_{i-1}T_iu_{i-1}^{s-1}u_{i-1}^{s^\prime-1}.
\end{align*}
\endgroup
\noindent Since $1\leq s-1\leq s^\prime-1\leq i-1$, by the induction hypothesis, $u_{i-1}^{s-1}u_{i-1}^{s^\prime-1}=(q-1)u_{i-2}^{s-2}u_{i-1}^{s^\prime-1}+qu_{i-2}^{s^\prime-2}u_{i-1}^{s-2}$. Using this in the previous equation and thereafter using the first relation in $H_{i+1}(q)$, we get \begingroup
\allowdisplaybreaks
\begin{align*}
u_i^su_i^{s^\prime}&=(q-1)T_{i-1}T_iu_{i-2}^{s-2}u_{i-1}^{s^\prime-1}+qT_{i-1}T_iu_{i-2}^{s^\prime-2}u_{i-1}^{s-2}\\
&=(q-1)T_{i-1}u_{i-2}^{s-2}T_iu_{i-1}^{s^\prime-1}+qT_{i-1}u_{i-2}^{s^\prime-2}T_iu_{i-1}^{s-2}\\
&=(q-1)u_{i-1}^{s-1}u_i^{s^\prime}+qu_{i-1}^{s^\prime-1}u_i^{s-1}.
\end{align*}
\endgroup
\end{enumerate}
\item To prove the second equality: We use the induction on $i$ to prove the equality. The condition $1\leq s^\prime<s\leq i$ is not true for $i=1$. For $i=2$, the condition $1\leq s^\prime<s\leq i$ implies that $s^\prime=1$ and $s=2$; thus, by the second relation in $H_3(q)$, $u_2^su_2^{s^\prime}=T_2T_1T_2=T_1T_2T_1=u_1^{s^\prime}u_2^s$. The equality is true for $i=2$. Assume that the equality holds for $i-1$ for $i\geq3$. We prove the equality for the next integer. Consider a product $u_i^su_i^{s^\prime}$ for $1\leq s^\prime<s\leq i$. We have the following cases:
\begin{enumerate}[(a)]
\item If $s^\prime=1$: By the first and second relations in $H_{i+1}(q)$,
\begin{equation*}
u_i^su_i^{s^\prime}=T_iT_{i-1}u_{i-2}^{s-2}T_i=T_iT_{i-1}T_iu_{i-2}^{s-2}=T_{i-1}T_iT_{i-1}u_{i-2}^{s-2}=u_{i-1}^{s^\prime}u_i^s.
\end{equation*}
\item If $s^\prime\geq2$: By the first and second relations in $H_{i+1}(q)$,
\begingroup
\allowdisplaybreaks
\begin{align*} u_i^su_i^{s^\prime}&=T_iT_{i-1}u_{i-2}^{s-2}T_iu_{i-1}^{s^\prime-1}=T_iT_{i-1}T_iu_{i-2}^{s-2}u_{i-1}^{s^\prime-1}\\
&=T_{i-1}T_iT_{i-1}u_{i-2}^{s-2}u_{i-1}^{s^\prime-1}=T_{i-1}T_iu_{i-1}^{s-1}u_{i-1}^{s^\prime-1}.
\end{align*}
\endgroup
\noindent Since $1\leq s^\prime-1<s-1\leq i-1$, by the induction hypothesis, $u_{i-1}^{s-1}u_{i-1}^{s^\prime-1}=u_{i-2}^{s^\prime-1}u_{i-1}^{s-1}$. Using this in the previous equation and thereafter using the first relation in $H_{i+1}(q)$, we get
\begin{equation*}
u_i^su_i^{s^\prime}=T_{i-1}T_iu_{i-2}^{s^\prime-1}u_{i-1}^{s-1}=T_{i-1}u_{i-2}^{s^\prime-1}T_iu_{i-1}^{s-1}=u_{i-1}^{s^\prime}u_i^s.\qedhere
\end{equation*}
\end{enumerate}
\end{enumerate}
\end{proof}

\begin{proposition}\label{prop3}
Let $i$ be a positive integer, $l\in\mathcal{L}_i$ and $m\in\mathcal{M}_i$. Then there is an algorithm to write the product $\beta_i^m\beta_i^l$ as a linear combination of elements in the basis $\mathcal{B}_i$ of $H_{i+1}(q)$ (i.e. as a linear combination of elements $\beta_i^{l^\prime}$ for $l^\prime\in\mathcal{L}_i$).
\end{proposition}

\begin{proof}
We prove the proposition by induction on $i$. Let $l\in\mathcal{L}_1$ and $m\in\mathcal{M}_1$. Since $\beta_1^l,\beta_1^m\in\{1,T_1\}$, ${T_1}^2=(q-1)T_1+q$ and $\mathcal{B}_1=\{1,T_1\}$, we get $\beta_1^m\beta_1^l\in\text{Span}(\mathcal{B}_1)$. Thus, the proposition is true for $i=1$. Assume that the proposition holds for a positive integer $i-1$. We prove the proposition for the next integer. Consider a product $\beta_i^m\beta_i^l$ for $l=(l_1,l_2,\ldots,l_i)$ in $\mathcal{L}_i$ and $m=(m_1,m_2,\ldots,m_i)$ in $\mathcal{M}_i$. By the induction hypothesis, $\beta_{i-1}^{m^{\prime}}\beta_{i-1}^{l^{\prime}}\in\text{Span}\left(\mathcal{B}_{i-1}\right)$ for $l^{\prime}\in\mathcal{L}_{i-1}$ and $m^{\prime}\in\mathcal{M}_{i-1}$; thus, it is sufficient to prove that
\begin{equation*}
\beta_i^m\beta_i^l\in\text{Span}\Big\{\beta_{i-1}^{m^{\prime}}\beta_{i-1}^{l^{\prime}}u_i^s\mid l^{\prime}\in\mathcal{L}_{i-1},\;m^{\prime}\in\mathcal{M}_{i-1}\;\text{and}\;0\leq s\leq i\Big\}.
\end{equation*} 
\noindent We can write $\beta_i^m\beta_i^l=\beta_{i-1}^{m^\prime}u_i^{m_i}\beta_i^l$, where $m=\left(m^{\prime},m_i\right)$. Depending on $m_i$ is $0$ or $1$, we have the following cases:
\begin{enumerate}[(1)]
\item If $m_i=0$: We get $\beta_i^m\beta_i^l=\beta_{i-1}^{m^\prime}\beta_{i-1}^{l^\prime}u_i^{l_i}$, where $l=\left(l^{\prime},l_i\right)$.
\item If $m_i=1$: By the first relation in $H_{i+1}(q)$,
\begin{equation*} \beta_i^m\beta_i^l=\beta_{i-1}^{m^\prime}T_i\beta_i^l=\beta_{i-1}^{m^\prime}T_i\beta_{i-2}^{l^{\prime\prime}}u_{i-1}^{l_{i-1}}u_i^{l_i}=\beta_{i-1}^{m^\prime}\beta_{i-2}^{l^{\prime\prime}}u_i^{l_{i-1}+1}u_i^{l_i},
\end{equation*}
\noindent where $l=\left(l^{\prime\prime},l_{i-1},l_i\right)$. According to some conditions on $l_{i-1}$ and $l_i$, we have the following cases:
\begin{enumerate}[(a)]
\item If $l_i=0$: We get $\beta_i^m\beta_i^l=\beta_{i-1}^{m^\prime}\beta_{i-2}^{l^{\prime\prime}}u_i^{l_{i-1}+1}=\beta_{i-1}^{m^\prime}\beta_{i-1}^{l^{\prime\prime\prime}}u_i^{l_{i-1}+1}$, where $l^{\prime\prime\prime}=\left(l^{\prime\prime},0\right)$.
\item If $l_{i-1}<l_i$: By Lemma \ref{lem3} (1),
\begingroup
\allowdisplaybreaks
\begin{align*} \beta_i^m\beta_i^l&=(q-1)\beta_{i-1}^{m^\prime}\beta_{i-2}^{l^{\prime\prime}}u_{i-1}^{l_{i-1}}u_i^{l_i}+q\beta_{i-1}^{m^\prime}\beta_{i-2}^{l^{\prime\prime}}u_{i-1}^{l_i-1}u_i^{l_{i-1}}\\
&=(q-1)\beta_{i-1}^{m^\prime}\beta_{i-1}^{l^\prime}u_i^{l_i}+q\beta_{i-1}^{m^\prime}\beta_{i-1}^{l^{\prime\prime\prime}}u_i^{l_{i-1}},
\end{align*}
\endgroup
\noindent where $l^{\prime}=\left(l^{\prime\prime},l_{i-1}\right)$ and $l^{\prime\prime\prime}=\left(l^{\prime\prime},l_i-1\right)$.
\item If $1\leq l_i\leq l_{i-1}$: By Lemma \ref{lem3} (2), $\beta_i^m\beta_i^l=\beta_{i-1}^{m^\prime}\beta_{i-2}^{l^{\prime\prime}}u_{i-1}^{l_i}u_i^{l_{i-1}+1}=\beta_{i-1}^{m^\prime}\beta_{i-1}^{l^{\prime\prime\prime}}u_i^{l_{i-1}+1}$, where $l^{\prime\prime\prime}=\left(l^{\prime\prime},l_i\right)$.\qedhere
\end{enumerate}
\end{enumerate} 
\end{proof}

\begin{lemma}\label{lem4}
Let $i$ be a positive integer, $l\in\mathcal{L}_i$ and $0\leq s\leq i$. Then there is an algorithm to write the product $\beta_i^lu_i^s$ as a linear combination of elements in the basis $\mathcal{B}_i$ of $H_{i+1}(q)$ (i.e. as a linear combination of elements $\beta_i^{l^\prime}$ for $l^\prime\in\mathcal{L}_i$).
\end{lemma}

\begin{proof}
We prove the lemma by induction on $i$. Let $l\in\mathcal{L}_1$ and $0\leq s\leq1$. Since $u_1^s,\beta_1^l\in\{1,T_1\}$, ${T_1}^2=(q-1)T_1+q$ and $\mathcal{B}_1=\{1,T_1\}$, we get $\beta_1^lu_1^s\in\text{Span}(\mathcal{B}_1)$. The lemma holds for $i=1$. Assume the lemma for a positive integer $i-1$. We prove the lemma for the next integer. Consider a product $\beta_i^lu_i^s$ for $l=\left(l_1,l_2,\ldots,l_i\right)$ in $\mathcal{L}_i$ and $0\leq s\leq i$. We have the following cases:
\begin{enumerate}[(1)]
\item If $s=0$: The product $\beta_i^lu_i^s$ is nothing but $\beta_i^l$.
\item If $s\neq0$: We can write $\beta_i^lu_i^s=\beta_{i-1}^{l^\prime}u_i^{l_i}u_i^s$, where $l=\left(l^\prime,l_i\right)$. Depending on some conditions on $l_i$, we have the following cases:
\begin{enumerate}[(a)]
\item If $l_i=0$: We get $\beta_i^lu_i^s=\beta_{i-1}^{l^\prime}u_i^s=\beta_i^{l^{\prime\prime}}$, where $l^{\prime\prime}=\left(l^\prime,s\right)$.
\item If $1\leq l_i\leq s$: By Lemma \ref{lem3} (1), we get
\begin{equation*}
\beta_i^lu_i^s=(q-1)\beta_{i-1}^{l^\prime}u_{i-1}^{l_i-1}u_i^s+q\beta_{i-1}^{l^\prime}u_{i-1}^{s-1}u_i^{l_i-1}.
\end{equation*}
\noindent Since by the induction hypothesis $\beta_{i-1}^{l^\prime}u_{i-1}^{l_i-1}$ and $\beta_{i-1}^{l^\prime}u_{i-1}^{s-1}$ belong to $\text{Span}(\mathcal{B}_{i-1})$, the product $\beta_i^lu_i^s$ belongs to $\text{Span}(\mathcal{B}_i)$.
\item If $s<l_i$: By Lemma \ref{lem3} (2), we get $\beta_i^lu_i^s=\beta_{i-1}^{l^\prime}u_{i-1}^su_i^{l_i}$. Since by the induction hypothesis $\beta_{i-1}^{l^\prime}u_{i-1}^s$ belongs to $\text{Span}(\mathcal{B}_{i-1})$, the product $\beta_i^lu_i^s$ belongs to $\text{Span}(\mathcal{B}_i)$.\qedhere
\end{enumerate}
\end{enumerate}
\end{proof}

\begin{proposition}\label{prop4}
There is an algorithm to compute the trace of elements in the basis $\mathcal{B}_i$ of $H_{i+1}(q)$ (i.e. the trace of elements $\beta_i^l$ for $l\in\mathcal{L}_i$).
\end{proposition}

\begin{proof}
We prove the proposition by induction on $i$. Note that $\mathcal{B}_1=\left\{1,T_1\right\}$. By using the properties of the trace, one can see that $Tr(1)=1$ and $Tr(T_1)=z$. Assume the proposition for a positive integer $i-1$. We prove the proposition for the next integer. Consider a basis element $\beta_i^l$ for $l=(l_1,l_2,\ldots,l_i)$ in $\mathcal{L}_i$. We have the following cases:
\begin{enumerate}[(1)]
\item If $l_i=0$: See that $\beta_i^l=\beta_{i-1}^{l^\prime}$, where $l=\left(l^\prime,0\right)$. Using the induction hypothesis, we can compute $Tr\!\left(\beta_{i-1}^{l^\prime}\right)$ which is nothing but $Tr\!\left(\beta_i^l\right)$.
\item If $l_i\neq0$: By the third property of the trace, 
\begin{equation*}
Tr\!\left(\beta_i^l\right)=Tr\!\left(\beta_{i-1}^{l^\prime}T_iu_{i-1}^{l_i-1}\right)=z\,Tr\!\left(\beta_{i-1}^{l^\prime}u_{i-1}^{l_i-1}\right),
\end{equation*}
\noindent where $l=\left(l^\prime,l_i\right)$. By Lemma \ref{lem4}, there is an algorithm to write the product $\beta_{i-1}^{l^\prime}u_{i-1}^{l_i-1}$ as a linear combination of elements in $\mathcal{B}_{i-1}$. By the linearity of the trace and using the induction hypothesis, one can compute $Tr\!\left(\beta_{i-1}^{l^\prime}u_{i-1}^{l_i-1}\right)$ and hence $Tr\!\left(\beta_i^l\right)$.\qedhere
\end{enumerate}
\end{proof}

Recall that $\rho:B_{n+1}\to H_{n+1}(q)$ defined by $\rho(\sigma_i)=T_i$, for $i=1,2,\ldots,n$, is a representation of $B_{n+1}$ into the multiplicative monoid of $H_{n+1}(q)$. Let $\alpha$ be a braid in $B_{n+1}$. If we can provide an algorithm to express $\rho(\alpha)$ as a linear combination of elements in $\mathcal{B}_n$, then Proposition \ref{prop4} (together with the linearity of the trace) can be used to compute the trace of $\rho(\alpha)$. In this direction, we have the following result.

\begin{proposition}
Let $\alpha$ be a braid in $B_{n+1}$. Then there is an algorithm to compute the trace of the element $\rho(\alpha)$ in $H_{n+1}(q)$.
\end{proposition}

\begin{proof}
Note that $\text{Tr}(\rho(\alpha))=1$ if $\alpha$ is the trivial braid (a braid with no crossings in its diagram). Let $\alpha$ be a nontrivial braid and let it be given by $\sigma_{i_1}^{\,\,j_1}\sigma_{i_2}^{\,\,j_2}\cdots\sigma_{i_p}^{\,\,j_p}$ (for $p\geq1$, $i_1,i_2,\ldots,i_p\in\{1,2,\ldots,n\}$ and $j_1,j_2,\ldots,j_p\in\mathbb{Z}$) as a product of powers of generators $\sigma_1,\sigma_2,\ldots,\sigma_n$ of $B_{n+1}$. We can assume that $i_{k-1}\neq i_k$ for $2\leq k\leq p$ and $j_k\neq0$ for $1\leq k\leq p$. Let $2\leq p_1<p_2<\cdots<p_r\leq p$ be the integers such that $i_{k-1}>i_k$ if $k\in\{p_1,p_2,\ldots,p_r\}$ and $i_{k-1}<i_k$ if $k\notin\{p_1,p_2,\ldots,p_r\}$. One can write
\begin{equation}\label{eq1}
\alpha=\left(\sigma_{i_1}^{\,\,j_1}\sigma_{i_2}^{\,\,j_2}\cdots\sigma_{i_{p_1-1}}^{\,\,j_{p_1-1}}\right)\left(\sigma_{i_{p_1}}^{\,\,j_{p_1}}\sigma_{i_{p_1+1}}^{\,\,j_{p_1+1}}\cdots\sigma_{i_{p_2-1}}^{\,\,j_{p_2-1}}\right)\cdots\left(\sigma_{i_{p_r}}^{\,\,j_{p_r}}\sigma_{i_{p_r+1}}^{\,\,j_{p_r+1}}\cdots\sigma_{i_p}^{\,\,j_p}\right)
\end{equation}
\noindent In \eqref{eq1}, there are $r+1$ parentheses. The subscript of $\sigma$ increases as one goes from left to right within a parenthesis. The subscript decreases as one jumps from a parenthesis into the next parenthesis. One can put zero powers of $\sigma_1,\sigma_2,\ldots,\sigma_n$ wherever necessary to make each parenthesis look like $\sigma_1^{\,\,k_1}\sigma_2^{\,\,k_2}\cdots\sigma_n^{\,\,k_n}$ for some integers $k_1,k_1,\ldots,k_n$. Thus \eqref{eq1} takes the form
\begin{equation}
\alpha=\left(\sigma_1^{\,\,k_{01}}\sigma_2^{\,\,k_{02}}\cdots\sigma_n^{\,\,k_{0n}}\right)\left(\sigma_1^{\,\,k_{11}}\sigma_2^{\,\,k_{12}}\cdots\sigma_n^{\,\,k_{1n}}\right)\cdots\left(\sigma_1^{\,\,k_{r1}}\sigma_2^{\,\,k_{r2}}\cdots\sigma_n^{\,\,k_{rn}}\right)
\end{equation}
\noindent where $k_{st}\in\mathbb{Z}$ for $0\leq s\leq r$ and $1\leq t\leq n$. Since $\rho$ is a homomorphism, we get
\begin{equation}\label{eq2}
\rho(\alpha)=\left(T_1^{\,\,k_{01}}T_2^{\,\,k_{02}}\cdots T_n^{\,\,k_{0n}}\right)\left(T_1^{\,\,k_{11}}T_2^{\,\,k_{12}}\cdots T_n^{\,\,k_{1n}}\right)\cdots\left(T_1^{\,\,k_{r1}}T_2^{\,\,k_{r2}}\cdots T_n^{\,\,k_{rn}}\right)
\end{equation}
\noindent By Proposition \ref{prop1}, each parenthesis in \eqref{eq2} can be written as a linear combination of elements that belong to $\mathcal{C}_n$. In other words, $T_1^{\,\,k_{s1}}T_2^{\,\,k_{s2}}\cdots T_n^{\,\,k_{sn}}=\sum_{m\in\mathcal{M}_n}Q_m^s(q)\beta_n^m$ for $0\leq s\leq r$, where $Q_m^s$ (for $0\leq s\leq r$ and $m\in\mathcal{M}_n$) is a Laurent polynomial in $q$. Using this in \eqref{eq2}, we get
\begingroup
\allowdisplaybreaks
\begin{align}
\rho(\alpha)&=\left(\sum_{m_0\in\mathcal{M}_n}Q_{m_0}^0(q)\beta_n^{m_0}\right)\left(\sum_{m_1\in\mathcal{M}_n}Q_{m_1}^1(q)\beta_n^{m_1}\right)\cdots\left(\sum_{m_r\in\mathcal{M}_n}Q_{m_r}^r(q)\beta_n^{m_r}\right)\\[2mm]
&=\sum_{m_0\in\mathcal{M}_n}\sum_{m_1\in\mathcal{M}_n}\cdots\sum_{m_r\in\mathcal{M}_n}Q_{m_0}^0(q)Q_{m_1}^1(q)\cdots Q_{m_r}^r(q)\beta_n^{m_0}\beta_n^{m_1}\cdots\beta_n^{m_r}\label{eq3}
\end{align}
\endgroup
\noindent For $m_0,m_1,\ldots,m_r\in\mathcal{M}_n$, using Proposition \ref{prop3} together with the induction on $r$, the product $\beta_n^{m_0}\beta_n^{m_1}\cdots\beta_n^{m_r}$ can be written as a linear combination of elements in the basis $\mathcal{B}_n$ of $H_{n+1}(q)$. Using this in \eqref{eq3}, we can express $\rho(\alpha)$ as a linear combination of basis elements. Since there is an algorithm (by Proposition \ref{prop4}) to compute the trace of the basis elements, we can compute the trace of $\rho(\alpha)$ by using the linearity of the trace.
\end{proof}

In this paper, our interest is in a special class of braids known as weaving braids. For a pair $(N,M)$ of positive integers, a {\it weaving braid} of type $(N,M)$, denoted by $\sigma_{N,M}$, is the braid $\left(\sigma_1\sigma_2^{-1}\sigma_3\sigma_4^{-1}\cdots\sigma_{N-1}^\delta\right)^m$, where $\delta=1$ if $N$ is even and $\delta=-1$ if $N$ is odd. The closure of the braid $\sigma_{N,M}$ is the knot or link $W(N,M)$ known as a {\it weaving knot} of type $(N,M)$ having the components equal to $\gcd(N,M)$. A picture of $W(4,5)$ is shown in Figure \ref{fig1}. The weaving knots are alternating and they form an interesting family of hyperbolic knots and links that have been studied earlier (for example, see \cite{ckp1,ckp2,ms}).

\begin{figure}[H]
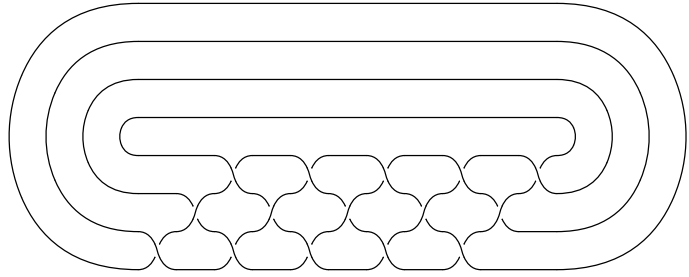

\begin{center}
\begin{equation*}
\xygraph{ !{0;/r1.2pc/:}
!{\hcap[1]}[u]
!{\hcap[3]}[u]
!{\hcap[5]}[u]
!{\hcap[7]}[lllllllllll]
!{\xcaph[-11]@(0)}[dl]
!{\xcaph[-11]@(0)}[dl]
!{\xcaph[-11]@(0)}[dl]
!{\xcaph[-11]@(0)}[uuul]
!{\hcap[-7]}[d]
!{\hcap[-5]}[d]
!{\hcap[-3]}[d]
!{\hcap[-1]}[d]
!{\xcaph[2]@(0)}[dl]
!{\xcaph[1]@(0)}[dl]
!{\htwist}[d]
!{\xcaph[1]@(0)}[uul]
!{\htwistneg}[u]
!{\htwist}[ddl]
!{\htwist}[d]
!{\xcaph[1]@(0)}[uul]
!{\htwistneg}[ul]
!{\xcaph[1]@(0)}
!{\htwist}[ddl]
!{\htwist}[d]
!{\xcaph[1]@(0)}[uul]
!{\htwistneg}[ul]
!{\xcaph[1]@(0)}
!{\htwist}[ddl]
!{\htwist}[d]
!{\xcaph[1]@(0)}[uul]
!{\htwistneg}[ul]
!{\xcaph[1]@(0)}
!{\htwist}[ddl]
!{\htwist}[d]
!{\xcaph[2]@(0)}[uul]
!{\htwistneg}[ul]
!{\xcaph[1]@(0)}
!{\htwist}[ddl]
!{\xcaph[1]@(0)}}
\end{equation*}\vspace{-38mm}
\caption{The weaving knot $W(4,5)$}
\label{fig1}
\end{center}
\end{figure}

We prefer to write $N=n+1$ and $M=m$, and study $W(n+1,m)$.

\begin{theorem}\label{thm1}
Let $\sigma_{n+1,m}$ be the weaving braid $\left(\sigma_1\sigma_2^{-1}\sigma_3\sigma_4^{-1}\cdots\sigma_n^\delta\right)^m$, where $\delta=1$ if $n$ is odd and $\delta=-1$ if $n$ is even. Then there is an algorithm to compute the trace of the element $\rho\left(\sigma_{n+1,m}\right)$ in $H_{n+1}(q)$.
\end{theorem}

\begin{proof}
Let $d=\frac{1-(-1)^n}{2}$ and $r=\frac{n-d}{2}$\,. There are exactly $r$ number of $\sigma_i$'s in $\sigma_1\sigma_2^{-1}\sigma_3\sigma_4^{-1}\cdots\sigma_n^\delta$ with power $-1$. Since ${T_i}^{-1}=q^{-1}\left(T_i+(1-q)\right)$ for $i=1,2,\ldots,n$, we have the following:
\begingroup
\allowdisplaybreaks
\begin{align}
\rho\left(\sigma_{n+1,1}\right)&=T_1T_2^{\,-1}T_3T_4^{\,-1}\cdots T_n^{\,\delta}\\
&=\left\{
\begin{array}{ll}
q^{-r}T_1\left(T_2+(1-q)\right)\cdots T_{n-1}\left(T_n+(1-q)\right) & \mbox{if}\;n\;\mbox{is even}\\[5pt]
q^{-r}T_1\left(T_2+(1-q)\right)\cdots \left(T_{n-1}+(1-q)\right)T_n & \mbox{if}\;n\;\mbox{is odd}
\end{array}\right.\\
&=q^{-r}\sum_{l\in\mathcal{N}_n}(1-q)^{n-\sum_sl_s}\beta_n^l\label{eq4}
\end{align}
\endgroup
\noindent Here, we recall that the length of $\beta_n^l$ (as a word in $T_k$'s) is $\sum_sl_s$ for $l=(l_1,l_2,\ldots,l_n)$ in $\mathcal{N}_n$. We can write \eqref{eq4} as follows:
\begin{equation}\label{eq5}
\rho\left(\sigma_{n+1,1}\right)=q^{-r}\sum_{l\in\mathcal{L}_n}f_l^1(q)\beta_n^l
\end{equation}
\noindent where
\begin{equation}\label{eq6}
f_l^1(q) =
\left\{
\begin{array}{ll}
(1-q)^{n-\sum_sl_s} & \mbox{if}\;l\in\mathcal{N}_n\\[5pt]
0 & \mbox{if}\;l\notin\mathcal{N}_n
\end{array}
\right.
\end{equation}
\noindent For $m\geq2$, assume that
\begin{equation}\label{eq24}
\rho\left(\sigma_{n+1,m-1}\right)=q^{-(m-1)r}\sum_{l\in\mathcal{L}_n}f_l^{m-1}(q)\beta_n^l
\end{equation}
\noindent for some polynomials $f_l^{m-1}$ in $q$. Let $i\in\mathcal{L}_n$ and $j\in\mathcal{N}_n$. Since $\mathcal{N}_n\subseteq \mathcal{M}_n$, $j\in\mathcal{M}_n$. Using the algorithm given in Proposition \ref{prop3} to write $\beta_n^j\beta_n^i$ as a linear combination of basis elements, we can find the polynomials $h_l^{ji}$ (for $l\in\mathcal{L}_n$) such that $\beta_n^j\beta_n^i=\sum_{l\in\mathcal{L}_n}h_l^{ji}(q)\beta_n^l$. By \eqref{eq4} and \eqref{eq24}, we have
\begingroup
\allowdisplaybreaks
\begin{align}	
\rho\left(\sigma_{n+1,m}\right)&=\rho\left(\sigma_{n+1,1}\right)\rho\left(\sigma_{n+1,m-1}\right)\\
&=q^{-mr}\left(\sum_{j\in\mathcal{N}_n}(1-q)^{n-\sum_sj_s}\beta_n^j\right)\left(\sum_{i\in\mathcal{L}_n}f_i^{m-1}(q)\beta_n^i\right)\\
&=q^{-mr}\sum_{i\in\mathcal{L}_n}\sum_{j\in\mathcal{N}_n}(1-q)^{n-\sum_sj_s}f_i^{m-1}(q)\beta_n^j\beta_n^i\\
&=q^{-mr}\sum_{i\in\mathcal{L}_n}\sum_{j\in\mathcal{N}_n}(1-q)^{n-\sum_sj_s}f_i^{m-1}(q)\sum_{l\in\mathcal{L}_n}h_l^{ji}(q)\beta_n^l\\
&=q^{-mr}\sum_{l\in\mathcal{L}_n}\left(\sum_{i\in\mathcal{L}_n}\sum_{j\in\mathcal{N}_n}(1-q)^{n-\sum_sj_s}f_i^{m-1}(q)h_l^{ji}(q)\right)\beta_n^l\\
&=q^{-mr}\sum_{l\in\mathcal{L}_n}f_l^m(q)\beta_n^l\label{eq7}
\end{align}
\endgroup
\noindent where
\begin{equation}\label{eq8}
f_l^m(q)=\sum_{i\in\mathcal{L}_n}\sum_{j\in\mathcal{N}_n}(1-q)^{n-\sum_sj_s}f_i^{m-1}(q)h_l^{ji}(q)
\end{equation}

\noindent Note that, using \eqref{eq8}, the polynomials $f_l^m$ (for $l\in\mathcal{L}_n$ and $m\geq2$) can be computed recursively. Since the trace function $Tr$ is linear and we have an algorithm (by Proposition \ref{prop4}) to compute the trace of basis elements, we can compute $Tr\left(\rho\left(\sigma_{n+1,m}\right)\right)$ using \eqref{eq7} if $m\geq2$ (or using \eqref{eq4} if $m=1$). This completes the proof of the theorem.
\end{proof}

We have written a Mathematica program to compute the trace of elements $\rho\left(\sigma_{n+1,m}\right)$ in $H_{n+1}(q)$ corresponding to weaving braids $\sigma_{n+1,m}$. The program as a Mathematica Notebook (.nb file) is available at \cite{mr}. We include the PDF file of the program in the next two pages.\vskip3mm

\noindent\hspace{-3pt}\includegraphics[width=1.01\linewidth]{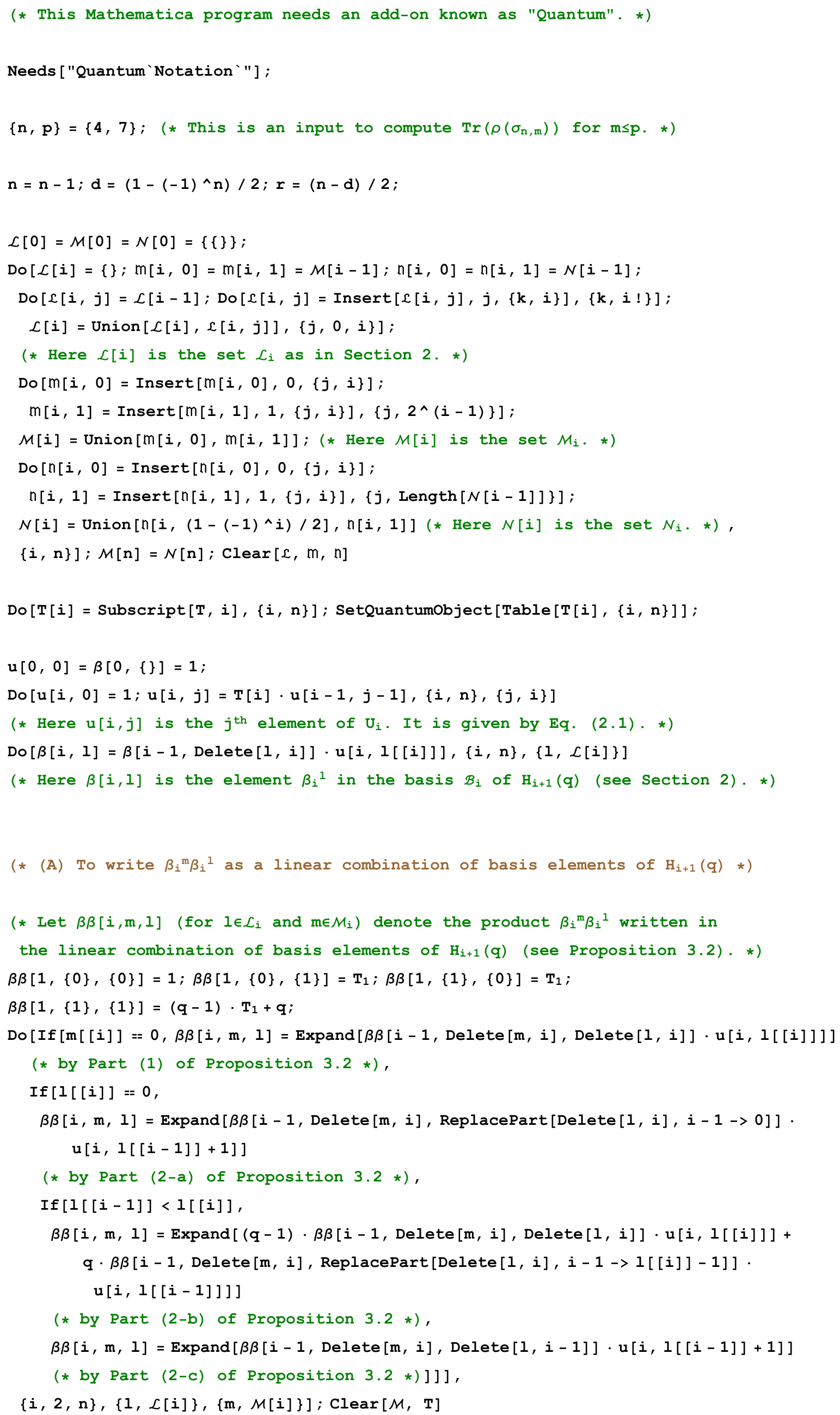}\vskip3mm
\noindent\hspace{-3pt}\includegraphics[width=1.01\linewidth]{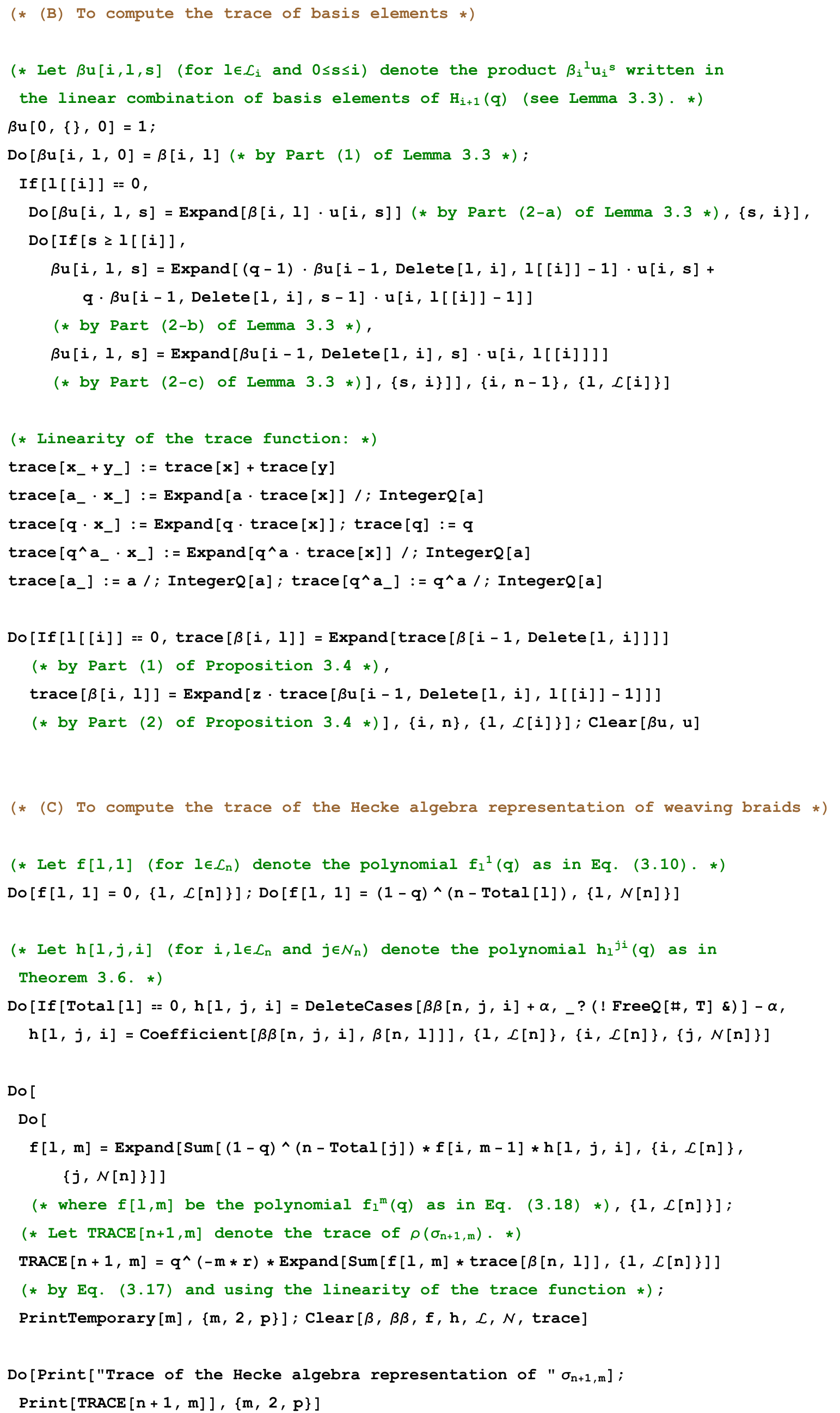}\vskip3mm

For a weaving braid $\sigma_{n+1,m}$ as in Theorem \ref{thm1}, we would like have an estimate on the degree as well as the lowest degree of $Tr\left(\rho\left(\sigma_{n+1,m}\right)\right)$ (note that this a Laurent polynomial in $q$ and $z$). For this, we first prove the following lemma that is going to be used in proving Proposition \ref{prop5} to get the desired estimate.

\begin{lemma}\label{lem5}
Let $\gamma$ be a word in $T_1,T_2,\ldots,T_n$. Then the degree of $Tr(\gamma)$ as a Laurent polynomial in $q$ and $z$ is same as the length of the word $\gamma$.
\end{lemma}

\begin{proof}
While computing the trace of a word in $T_k$'s, the first two relations in $H_{n+1}(q)$ and the first two properties of the trace function do not contribute a factor in the trace and hence in its leading term. These relations and properties do not affect the length of the word. The relation ${T_i}^2=(q-1)T_i+q$, for $1\leq i\leq n$, contributes a factor $q$ in the leading term by decreasing the length of the word by one. The property $Tr(aT_ib)=zTr(ab)$, for $a,b\in H_i(q)$ and $1\leq i\leq n$, contributes a factor $z$ in the leading term and the length of the word decreases by one. Keeping the above discussion in the mind, one can observe that the length of a word appears in the form of powers of $q$ and $z$ together in the leading term of the trace of the word. In other words, the leading term in the trace of any word is $q^iz^j$, where $i+j$ is the length of the word taken and $j$ is the number of distinct $T_k$'s appear in the word.
\end{proof}

\begin{proposition}\label{prop5}
For a weaving braid $\sigma_{n+1,m}$ as in Theorem \ref{thm1}, we have the following:
\begin{enumerate}[(1)]
\item Degree of $Tr\left(\rho\left(\sigma_{n+1,m}\right)\right)$ as a Laurent polynomial in $q$ and $z$ is equal to $mn-mr$, where $r=\frac{n-d}{2}$ and $d=\frac{1-(-1)^n}{2}$.
\item $Tr\left(\rho\left(\sigma_{n+1,m}\right)\right)$ as a Laurent polynomial in $q$ and $z$ has the lowest degree equal to $r+d-mr$.
\end{enumerate}
\end{proposition}

\begin{proof}
\begin{enumerate}[(1)]
\item Using \eqref{eq4}, one can obtain
\begin{equation}\label{eq25}
\rho\left(\sigma_{n+1,m}\right)=q^{-mr}\sum_{l^1\in\mathcal{N}_n}\sum_{l^2\in\mathcal{N}_n}\cdots\sum_{l^m\in\mathcal{N}_n}(1-q)^{mn-\sum_t\sum_sl_s^t}\beta_n^{l^1}\beta_n^{l^2}\cdots\beta_n^{l^m}
\end{equation}
\noindent where $l^t=\left(l_1^t,l_2^t,\ldots,l_n^t\right)$ for $0\leq t\leq m$. In \eqref{eq25}, a product $\beta_n^{l^1}\beta_n^{l^2}\cdots\beta_n^{l^m}$ is a word in $T_k$'s of length $\sum_t\sum_sl_s^t$. By Lemma \ref{lem5}, the leading term in the trace of $\beta_n^{l^1}\beta_n^{l^2}\cdots\beta_n^{l^m}$ will be $q^iz^j$, where $i+j=\sum_t\sum_sl_s^t$ and $j$ is the number of distinct $T_k$'s appear in $\beta_n^{l^1}\beta_n^{l^2}\cdots\beta_n^{l^m}$. Now the leading term in the trace of $(1-q)^{mn-\sum_t\sum_sl_s^t}\beta_n^{l^1}\beta_n^{l^2}\cdots\beta_n^{l^m}$ becomes $(-1)^{mn-\sum_t\sum_sl_s^t}q^iz^j$, where $i+j=mn$ and $j$ is the number of distinct $T_k$'s appear in $\beta_n^{l^1}\beta_n^{l^2}\cdots\beta_n^{l^m}$. 

The word $(T_1T_3\cdots T_{2r+c})^m$ is such that within it appear the least number of distinct $T_k$'s among all the words $\beta_n^{l^1}\beta_n^{l^2}\cdots\beta_n^{l^m}$ appearing in \eqref{eq25}, where $c$ is either $-1$ or $1$ depending on $n$ is even or odd respectively. Since the number of distinct $T_k$'s in $(T_1T_3\cdots T_{2r+c})^m$ are $r+d$, the leading term in the trace of $(1-q)^{mr}(T_1T_3\cdots T_{2r+c})^m$ will be $(-1)^{mr}q^iz^j$, where $i+j=mn$ and $j=r+d$. Now by looking at \eqref{eq25}, one can see that $(-1)^{mr}q^{i-mr}z^j$ is among the leading terms in the trace of $\rho\left(\sigma_{n+1,m}\right)$, where $i+j=mn$ and $j=r+d$. In other words, the degree of the trace of $\rho\left(\sigma_{n+1,m}\right)$ is $mn-mr$.

\item As noted in the proof of Lemma \ref{lem5}, while computing the trace of a word in $T_k$'s, the first two relations in $H_{n+1}(q)$ and the first two properties of the trace function do not contribute a factor in the trace and hence in its lowest degree term. These relations and properties do not affect the length of the word. The relation ${T_i}^2=(xy+x)T_i+xy$, where $x=-1$ and $y=-q$, contributes a factor $x$ in some lowest degree term and the length of the word decreases by one. The property $Tr(aT_ib)=zTr(ab)$, for $a,b\in H_i(q)$ and $1\leq i\leq n$, contributes a factor $z$ in a lowest degree term by decreasing the length of the word by one. Keeping the above discussion in the mind, one can observe that the length of a word appears in the form of powers of $x$ and $z$ together in some lowest degree term of the trace of the word. With this, one can see that one of the lowest degree term in the trace of $\beta_n^{l^1}\beta_n^{l^2}\cdots\beta_n^{l^m}$ (this is a word in $T_k$'s appearing in \eqref{eq25}) will be $x^iz^j=(-1)^iz^j$, where $i+j$ is the length of $\beta_n^{l^1}\beta_n^{l^2}\cdots\beta_n^{l^m}$ as a word in $T_k$'s and $j$ is the number of distinct $T_k$'s appear in $\beta_n^{l^1}\beta_n^{l^2}\cdots\beta_n^{l^m}$.

As mentioned in the first part, the word $(T_1T_3\cdots T_{2r+c})^m$ is such that within it appear the least number of distinct $T_k$'s among all the words $\beta_n^{l^1}\beta_n^{l^2}\cdots\beta_n^{l^m}$ appearing in \eqref{eq25}. The length of $(T_1T_3\cdots T_{2r+c})^m$ is $m(r+d)=mn-mr$ and the number of distinct $T_k$'s in this word are $r+d$. Thus, one of the lowest degree term in the trace of $(T_1T_3\cdots T_{2r+c})^m$ is $x^iz^j=(-1)^iz^j$ for $i+j=mn-mr$ and $j=r+d$. Finally, one of the lowest degree term in the trace of $\rho\left(\sigma_{n+1,m}\right)$ will be $(-1)^iq^{-mr}z^j$ for $i+j=mn-mr$ and $j=r+d$.\qedhere
\end{enumerate}
\end{proof}

\section{From trace to polynomial invariants}\label{sec4}

We use the construction given in \cite{hkw} (see also \cite{j2}) and work over the function field $K=\mathbb{C}(q,z)$ to obtain expressions for the Alexander polynomial \cite{a}, the Jones polynomial \cite{j1} and the HOMFLY-PT polynomial \cite{homfly,pt} for weaving knots. The expressions are subsequently refined to incorporate information obtained in Section \ref{sec3}.

Referring \cite[Section 6]{hkw}, the invariant $V_\alpha$ in variables $q$ and $z$ of a link $L$ which is the closure of a braid $\alpha\in B_{n+1}$ is given by
\begin{equation}\label{eq10}
V_\alpha(q,z)=\left(\frac{1}{z}\right)^{\frac{n+e}{2}}\left(\frac{q}{w}\right)^{\frac{n-e}{2}}Tr\!\left(\rho(\alpha)\right)
\end{equation}
\noindent where $w=1-q+z$ and $e$ is the exponent sum of $\alpha$ written as a word in $\sigma_i$'s. The \eqref{eq10} defines an element in the quadratic extension $K\!\left(\sqrt{q/zw}\right)$ of $K$. Now consider the invariant $X_L$ in variables $q$ and $\lambda$ as defined in \cite[Section 6]{j2}. It is given by
\begin{equation}
X_L(q,\lambda)=\left(-\frac{1-\lambda\,q}{\sqrt{\lambda}\,(1-q)}\right)^n\left(\sqrt{\lambda}\right)^{e}Tr(\pi(\alpha))
\end{equation}
\noindent where $\pi$ is the representation $\rho$ and $Tr(\pi(\alpha))$ is evaluated at $z=-\frac{1-q}{1-\lambda\,q}$\,. One can see that the invariant $X_L$ is nothing but $V_\alpha$ evaluated at $z=-\frac{1-q}{1-\lambda\,q}$\,.

For a weaving knot $W(n+1,m)$, viewed as the closure of $\sigma_{n+1,m}=\left(\sigma_1\sigma_2^{-1}\sigma_3\sigma_4^{-1}\cdots\sigma_n^\delta\right)^m$, we have the exponent sum $e=0$ if $n$ is even and $e=m$ if $n$ is odd. The invariant $V_{\sigma_{n+1,m}}$ (as defined in \eqref{eq10}) of $W(n+1,m)$ is given by
\begin{equation}\label{eq11}
V_{\sigma_{n+1,m}}(q,z)=\left(\frac{1}{z}\right)^{\frac{n+e}{2}}\left(\frac{q}{w}\right)^{\frac{n-e}{2}}Tr(\rho(\sigma_{n+1,m}))
\end{equation}
\noindent Recall that, using the algorithm in Section \ref{sec3}, we can compute $Tr(\rho(\sigma_{n+1,m}))$. In fact, this can be computed using the Mathematica program discussed in Section \ref{sec3}. Thus, using \eqref{eq11}, one can compute $V_{\sigma_{n+1,m}}(q,z)$.

Using \eqref{eq7} if $m\geq2$ (or using \eqref{eq5} if $m=1$) in \eqref{eq11} and using the linearity of the trace, we have the following:
\begin{equation}\label{eq12}
V_{\sigma_{n+1,m}}(q,z)=\left(\frac{1}{z}\right)^{\frac{n+e}{2}}\left(\frac{q}{w}\right)^{\frac{n-e}{2}}q^{-mr}\sum_{l\in\mathcal{L}_n}f_l^m(q)\,Tr\!\left(\beta_n^l\right)
\end{equation}
\noindent where polynomials $f_l^m$, for $m\geq2$, are given by \eqref{eq8} with the initial polynomials $f_l^1$ as in \eqref{eq6}.

Following \cite[Section 6]{hkw}, we point out that the universal skein invariant $P_{W(n+1,m)}(\ell, m)$, an element of the Laurent polynomial ring $\mathbb{Z}\!\left[\ell, \ell^{-1}, m , m^{-1}\right]$, is obtained by rewriting $V_{\sigma_{n+1,m}}(q,z)$ in terms of 
\begin{equation}\label{eq13}
\ell=i(z/w)^{1/2}\qquad\text{and}\qquad m=i\left(q^{-1/2}-q^{1/2}\right)
\end{equation}

Starting from $P_{W(n+1,m)}(\ell, m)$, the Alexander polynomial $\Delta_{W(n+1,m)}(t)$ is obtained by setting $\ell=i$ and $m=i\left(t^{1/2}-t^{-1/2}\right)$ (see \cite[Section 7]{hkw}). One can also obtain it by setting
\begin{equation}\label{eq14}
\ell=-i\qquad\text{and}\qquad m=i\left(t^{-1/2}-t^{1/2}\right)
\end{equation}
\noindent Solving \eqref{eq13} and \eqref{eq14} together, one gets $q=t$ and $z=\frac{t-1}{b}$, where $b=1+\ell^2$. The Alexander polynomial $\Delta_{W(n+1,m)}(t)$ is obtained by substituting $q=s+1$, $z=\frac{s}{b}$ and $w=\frac{s}{b}$ in \eqref{eq11}, simplifying the expression thereafter and then substituting $s=t-1$ and $b=1+(-i)^2=0$ at the last. With this, we have the following Mathematica program to compute the Alexander polynomial for weaving knots. This is in continuation with the program for the trace discussed in Section \ref{sec3}.\vskip3mm
\noindent\hspace{-3pt}\includegraphics[width=1.01\linewidth]{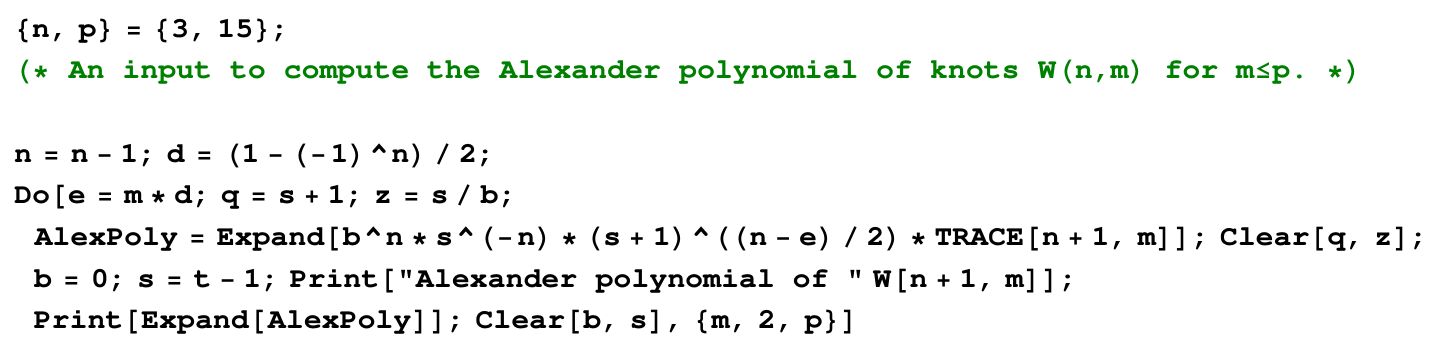}\vskip3mm
\noindent Below are some examples of the Alexander polynomial computed using this program.
\begin{itemize}
\item\(\Delta_{W(3,5)}(t)=29+\dfrac{1}{t^4}-\dfrac{6}{t^3}+\dfrac{15}{t^2}-\dfrac{24}{t}-24 t+15 t^2-6 t^3+t^4\)
\item\(\Delta_{W(5,2)}(t)=13+\dfrac{1}{t^2}-\dfrac{7}{t}-7 t+t^2\)
\item\(\Delta_{W(6,5)}(t)=79781+\frac{1}{t^{10}}-\frac{21}{t^9}+\frac{195}{t^8}-\frac{1075}{t^7}+\frac{4010}{t^6}-\frac{10989}{t^5}+\frac{23485}{t^4}-\frac{40871}{t^3}+\frac{59620}{t^2}-\frac{74245}{t}-74245t+59620 t^2-40871 t^3+23485 t^4-10989 t^5+4010 t^6-1075 t^7+195 t^8-21 t^9+t^{10}\)
\end{itemize}
\noindent Here, the weaving knots $W(3,5)$ and $W(5,2)$ are respectively the knots $10_{123}$ and $8_{12}$ in the Rolfsen Knot Table (\url{http://katlas.org/wiki/The_Rolfsen_Knot_Table}).

Referring \cite[Section 0]{lm}, the Jones polynomial $V_{W(n+1,m)}(t)$ is obtained from $P_{W(n+1,m)}(\ell, m)$ by setting
\begin{equation}\label{eq15}
\ell=it^{-1}\qquad\text{and}\qquad m=i\left(t^{-1/2}-t^{1/2}\right)
\end{equation}
\noindent Solving \eqref{eq13} and \eqref{eq15} together, we get $q=t$ and $z=\frac{-1}{1+t}$\,. Substituting $q=t$, $z=\frac{-1}{1+t}$ and $w=\frac{-t^2}{1+t}$ in \eqref{eq11}, we get the Jones polynomial $V_{W(n+1,m)}(t)$. With this, we have the following Mathematica program to compute the Jones polynomial for weaving knots. This works along with the program for the trace discussed in Section \ref{sec3}.\vskip3mm
\noindent\hspace{-3pt}\includegraphics[width=1.01\linewidth]{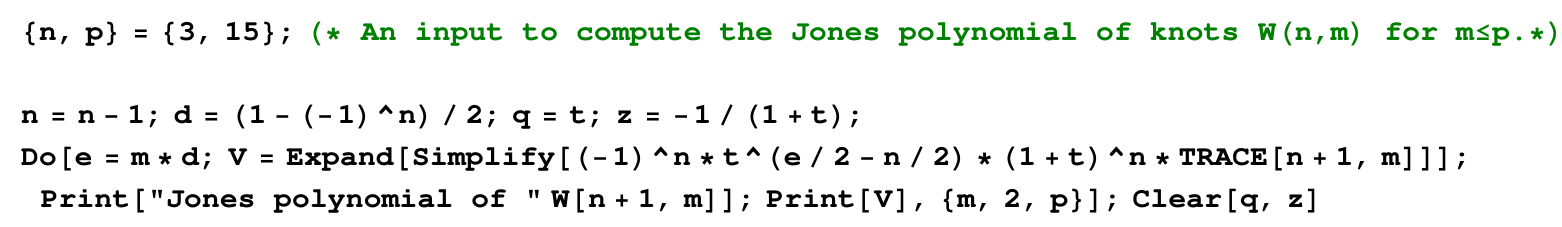}\vskip3mm
\noindent Below are some examples of the Jones polynomial computed using this program.
\begin{itemize}
\item\(V_{W(3,3)}(t)=4-\dfrac{1}{t^3}+\dfrac{3}{t^2}-\dfrac{2}{t}-2 t+3 t^2-t^3\)
\item\(V_{W(4,3)}(t)=-8-\dfrac{1}{t^2}+\dfrac{5}{t}+11 t-13 t^2+13 t^3-11 t^4+8 t^5-4 t^6+t^7\)
\item\(V_{W(6,4)}(t)=-\frac{1}{t^{13/2}}+\frac{11}{t^{11/2}}-\frac{58}{t^{9/2}}+\frac{200}{t^{7/2}}-\frac{519}{t^{5/2}}+\frac{1079}{t^{3/2}}-\frac{1869}{\sqrt{t}}+2776\sqrt{t}-3613 t^{3/2}+4177 t^{5/2}-4324 t^{7/2}+4018 t^{9/2}-3335 t^{11/2}+2453 t^{13/2}-1579 t^{15/2}+872 t^{17/2}-407 t^{19/2}+154 t^{21/2}-45t^{23/2}+9 t^{25/2}-t^{27/2}\)
\end{itemize}
\noindent Here, the weaving knot $W(3,3)$ is the link $L6a4$ in the Thistlethwaite Link Table (\url{http://katlas.org/wiki/The_Thistlethwaite_Link_Table}) and the weaving knot $W(4,3)$ is the mirror image of the knot $9_{40}$ in the Rolfsen Knot Table. One can see that $V_{9_{40}}(t)=V_{W(4,3)}\!\left(t^{-1}\right)$ which is the case in general for the mirror image of a knot. We note that $W(3,3)$ and $W(6,4)$ are links with $3$ and $2$ components respectively having the Jones polynomial as Laurent polynomials in $t$ and $t^{1/2}$ respectively.

The HOMFLY-PT polynomial $H_{W(n+1,m)}(\mathtt{a},\mathtt{z})$ is obtained from the polynomial $P_{W(n+1,m)}(\ell, m)$ by the change of variables as follows:
\begin{equation}\label{eq16}
\ell=i\mathtt{a}\qquad\text{and}\qquad m=i\mathtt{z}
\end{equation}
\noindent Solving \eqref{eq13} and \eqref{eq16} together gives $q=\frac{\mathtt{w}\mathtt{z}}{2}+1$ and $z=\frac{\mathtt{a}^2\mathtt{w}\mathtt{z}}{2\mathtt{b}}$, where $\mathtt{w}=\mathtt{z}+\sqrt{\mathtt{z}^2+4}$ and $\mathtt{b}=\mathtt{a}^2-1$\,. The HOMFLY-PT polynomial $H_{W(n+1,m)}(\mathtt{a},\mathtt{z})$ is obtained by substituting $q=\frac{\mathtt{w}\mathtt{z}}{2}+1$, $z=\frac{\mathtt{a}^2\mathtt{w}\mathtt{z}}{2\mathtt{b}}$ and $w=\frac{\mathtt{w}\mathtt{z}}{2\mathtt{b}}$ with $\mathtt{w}=\mathtt{z}+\sqrt{\mathtt{z}^2+4}$ in \eqref{eq11}, simplifying the expression thereafter and then substituting $\mathtt{b}=\mathtt{a}^2-1$ at the last. With this, we have the following Mathematica program to compute the HOMFLY-PT polynomial for weaving knots. This is in continuation with the program for the trace discussed in Section \ref{sec3}.\vskip3mm
\noindent\hspace{-3pt}\includegraphics[width=1.01\linewidth]{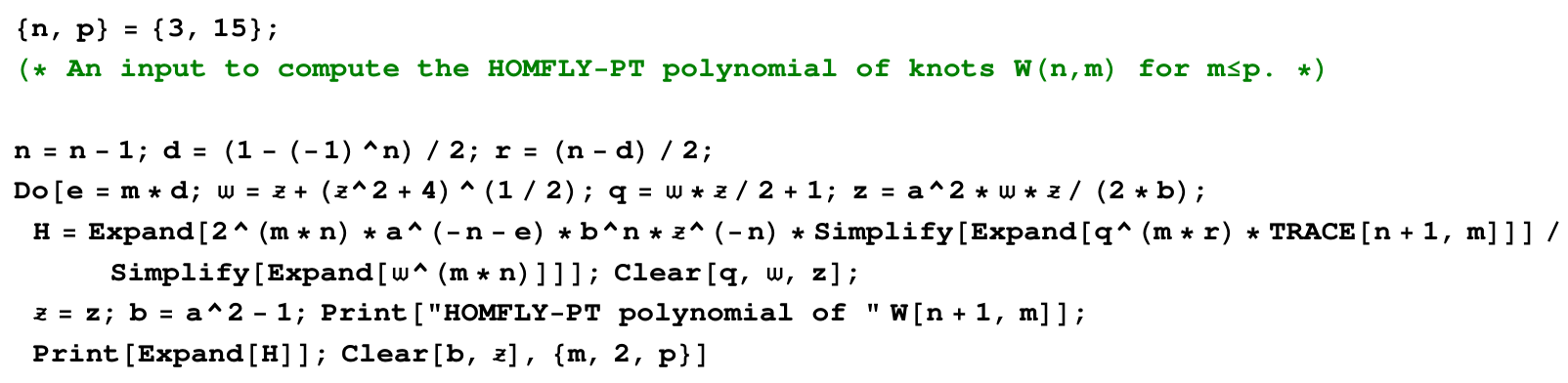}\vskip3mm
\noindent Below are some examples of the HOMFLY-PT polynomial computed using this program.
\begin{itemize}
\item\(H_{W(3,12)}(a,z)=-\frac{2}{z^2}+\frac{1}{a^2 z^2}+\frac{a^2}{z^2}+32 z^2-\frac{16 z^2}{a^2}-16 a^2 z^2+16 z^4-120 z^6+\frac{60 z^6}{a^2}+60 a^2 z^6-100
z^8+\frac{20 z^8}{a^2}+20 a^2 z^8+124 z^{10}-\frac{72 z^{10}}{a^2}-72 a^2 z^{10}+168 z^{12}-\frac{48 z^{12}}{a^2}-48 a^2 z^{12}+10 z^{14}+\frac{19
z^{14}}{a^2}+19 a^2 z^{14}-73 z^{16}+\frac{27 z^{16}}{a^2}+27 a^2 z^{16}-45 z^{18}+\frac{9 z^{18}}{a^2}+9 a^2 z^{18}-11 z^{20}+\frac{z^{20}}{a^2}+a^2
z^{20}-z^{22}\)
\item\(H_{W(4,2)}(a,z)=-\dfrac{1}{a^5 z}+\dfrac{1}{a^3 z}+\dfrac{2 z}{a^3}-\dfrac{z}{a}+a z-\dfrac{z^3}{a}\)
\item\(H_{W(5,2)}(a,z)=1+\dfrac{1}{a^4}-\dfrac{1}{a^2}-a^2+a^4+z^2-\dfrac{2 z^2}{a^2}-2 a^2 z^2+z^4\)
\end{itemize}
Here, the weaving knot $W(4,2)$ is the mirror image of the link $L6a1$ in the Thistlethwaite Link Table and the weaving knot $W(5,2)$ is the knot $8_{12}$ in the Rolfsen Knot Table. It can be seen that $H_{L6a1}(a,z)=H_{W(4,2)}\!\left(-a^{-1},z\right)$ which is the case in general for the mirror image of a knot. Note that $W(3,12)$ and $W(4,2)$ are links with $3$ and $2$ components respectively.

\section{Higher twist numbers versus volume}\label{sec5}

In \cite{dl}, Dasbach and Lin introduced higher twist numbers of a knot (link) $L$ in terms of the coefficients of its Jones polynomial $V_L(t)$ with the idea that these invariants correlate with the hyperbolic volume of the knot (link) complement. The twist numbers are defined as follows:
\begin{definition}
Let
\begin{equation*}
 V_L(t)=\lambda_{l}t^{l} + \lambda_{l+1}t^{l+1} + \cdots + \lambda_{h-1}t^{h-1} + \lambda_ht^h
\end{equation*}
\noindent be the Jones polynomial of a knot (link) $L$. Then the $j^{\text{th}}$ twist number $T_j(L)$ of $L$ is defined as $T_j(L)=|{\lambda_{l+j}}|+|{\lambda_{h-j}}|$\,.
\end{definition}

Note that twist numbers $T_j(L)$ are only defined for $j$ within the span of the Jones polynomial of $L$. In the case of a weaving knot $W(n+1,m)$, the relevant twist numbers are defined for $1\leq j<\frac{mn}{2}$. Here $mn$ is the span of the Jones polynomial of $W(n+1,m)$.

According to \cite{ckp2}, the volume of a weaving knot $W(n+1,m)$, for $n\geq 2$ and $m \geq 7$, is estimated as
\begin{equation}\label{eq17}
v_{\rm oct}\,(n-1)\,m \left(1-\frac{4\pi^2}{m^2}\right)^{3/2} \leq {\rm vol}\left(W(n+1,m)\right)<\left(v_{\rm oct}\,(n-2)+4\,v_{\rm tet}\right)m
\end{equation}
\noindent Here $v_{\rm tet}$ and $v_{\rm oct}$ denote the volumes of the regular ideal tetrahedron and the regular ideal octahedron respectively (note that $v_{\rm tet}\approx1.01494$ and $v_{\rm oct}\approx3.66386$). Champanerkar, Kofman, and Purcell call these bounds asymptotically sharp because their ratio approaches $1$ as $m$ and $n$ tend to infinity. The crossing number of $W(n+1,m)$ is known to be $mn$. Dividing this number throughout, the inequalities in \eqref{eq17} imply the following:
\begin{equation}\label{eq18}
\frac{v_{\rm oct}}{n}\,(n-1)\left(1-\frac{4\pi^2}{m^2}\right)^{3/2}\leq\frac{{\rm vol}\left(W(n+1,m)\right)}{mn}<\frac{v_{\rm oct}\,(n-2)+4\,v_{\rm tet}}{n}
\end{equation}
\noindent For a fixed $n$, note that the upper bound for the ratio $\frac{{\rm vol}\left(W(n+1,m)\right)}{mn}$ is constant for all $m$. Let $L_n=\frac{v_{\rm oct}}{n}(n-1)$ and $U_n=\frac{v_{\rm oct}(n-2)+4v_{\rm tet}}{n}$, and let $v_n(m)$ denote the relative volume $\frac{{\rm vol}\left(W(n+1,m)\right)}{mn}$ of the weaving knot $W(n+1,m)$. In these notations, one can write \eqref{eq18} as follows:
\begin{equation}\label{eq19}
L_n\left(1-\frac{4\pi^2}{m^2}\right)^{3/2}\leq v_n(m)<U_n
\end{equation}
\noindent By taking the limit infimum of both the sides of the first inequality and the limit supremum of both the sides of the second inequality, we get the following:
\begin{equation}\label{eq20}
L_n\leq\liminf_{m\to\infty}v_n(m)\leq\limsup_{m\to\infty}v_n(m)\leq U_n
\end{equation}

We can ask whether or not better bounds on the relative volume of weaving knots can be observed in terms of the higher twist numbers? In \cite{ms}, the authors took the $k^{\text{th}}$ root of $T_k(W(3,m))$ and then divided by the crossing number $2m$ to obtain an expression whose limit as $m$ tends to infinity is finite. Multiplying by a normalization constant $C_k$ so that
\begin{equation*}
\lim_{m\to\infty}C_k\,\frac{\sqrt[k]{T_k\!\left(W(3,m)\right)}}{2\:\!m} = 2\:\!v_{\rm tet}\,,
\end{equation*}
\noindent they plotted the lower bound $\frac{v_{\rm oct}}{2}\left(1-\frac{4\pi^2}{m^2}\right)^{3/2}$, the upper bound $2v_{\rm tet}$, the relative volume $v_2(m)$ and $\tau_k(m)=C_k\,\frac{\sqrt[k]{T_k\!\left(W(3,m)\right)}}{2\:\!m}$ for $k=2,3,4$. The authors showed that all three of the curves $\tau_k$ provide better lower bounds on the relative volume than the lower bound $\frac{v_{\rm oct}}{2}\left(1-\frac{4\pi^2}{m^2}\right)^{3/2}$. To explore more on this, we used the program (see Section \ref{sec4}) for the Jones polynomial and extended it to compute various twist numbers. We ran the program for $W(3,m)$, $W(4,m)$, $W(5,m)$ and $W(6,m)$ for various values of $m$. In an experiment we performed, we observed that
\begin{equation}\label{eq26}
\frac{T_k(m)}{\frac{\left((r+d)^k+r^k\right)m^k}{k!}}
\end{equation}
\noindent converges to $1$ as $m$ grows large, where $T_k(m)$ is the $k^{\text{th}}$ twist number of $W(n+1,m)$ for $n$ fixed, $d=\frac{1-(-1)^n}{2}$ and $r=\frac{n-d}{2}$\,. Let us denote \eqref{eq26} by $f_k(m)$ and its $k^{\text{th}}$ root by $g_k(m)$. The tables at the end (Table \ref{tb1} to Table \ref{tb2}) show the values of $f_k(m)$ for the families $W(3,m)$, $W(4,m)$, $W(5,m)$ and $W(6,m)$ for various values of $k$ and $m$. For every $k$, the data suggests that both $f_k(m)$ and $g_k(m)$ converge to $1$ as $m$ grows large. With this discussion, we conjuncture the following.
\begin{conjecture}
For a fixed $n$, let $T_k(m)$ be the $k^{\text{th}}$ twist number of $W(n+1,m)$. Then
\begin{equation*}
T_k(m)\approx\frac{\left((r+d)^k+r^k\right)m^k}{k!}\qquad\text{for $m$ large enough,}
\end{equation*}
\noindent where $d=\frac{1-(-1)^n}{2}$ and $r=\frac{n-d}{2}$.
\end{conjecture}

For a fixed $n$, consider the following functions:
\begingroup
\allowdisplaybreaks
\begin{align*}
L_k^1(m)&= L_n\left(1+|1-f_k(m)|\right)\qquad\qquad\quad\;\,\,L_k^2(m)=\frac{L_n}{1-|1-f_k(m)|}\\[3pt]
L_k^3(m)&= L_n\left(1+|1-g_k(m)|\right)\qquad\text{and}\qquad L_k^4(m)=\frac{L_n}{1-|1-g_k(m)|}\;.
\end{align*}
\endgroup
\noindent From our observations regarding $f_k(m)$ and $g_k(m)$, the sequence $L_k^i(m)$, for $i=1,2,3$ and $4$, is decreasing and $\lim_{m\to\infty}L_k^i(m)=L_n$. If the first inequality in \eqref{eq20} is strict, then one can observe that $L_k^i(m)\leq v_n(m)$ for $m$ large enough. Thus, each $L_k^i(m)$ seems to provide a better lower bound for the relative volume $v_n(m)$ than the lower bound as in \eqref{eq19}. As an example for the family $W(4,m)$, Figure \ref{fig2} displays the graphs of $L_k^3$ for various values of $k$ along with the graphs of bounds on relative volume as in \eqref{eq19} and the values of relative volume computed using SnapPy \cite{cdgw}.

\begin{figure}[H]
\begin{center}
\includegraphics[width=1\linewidth]{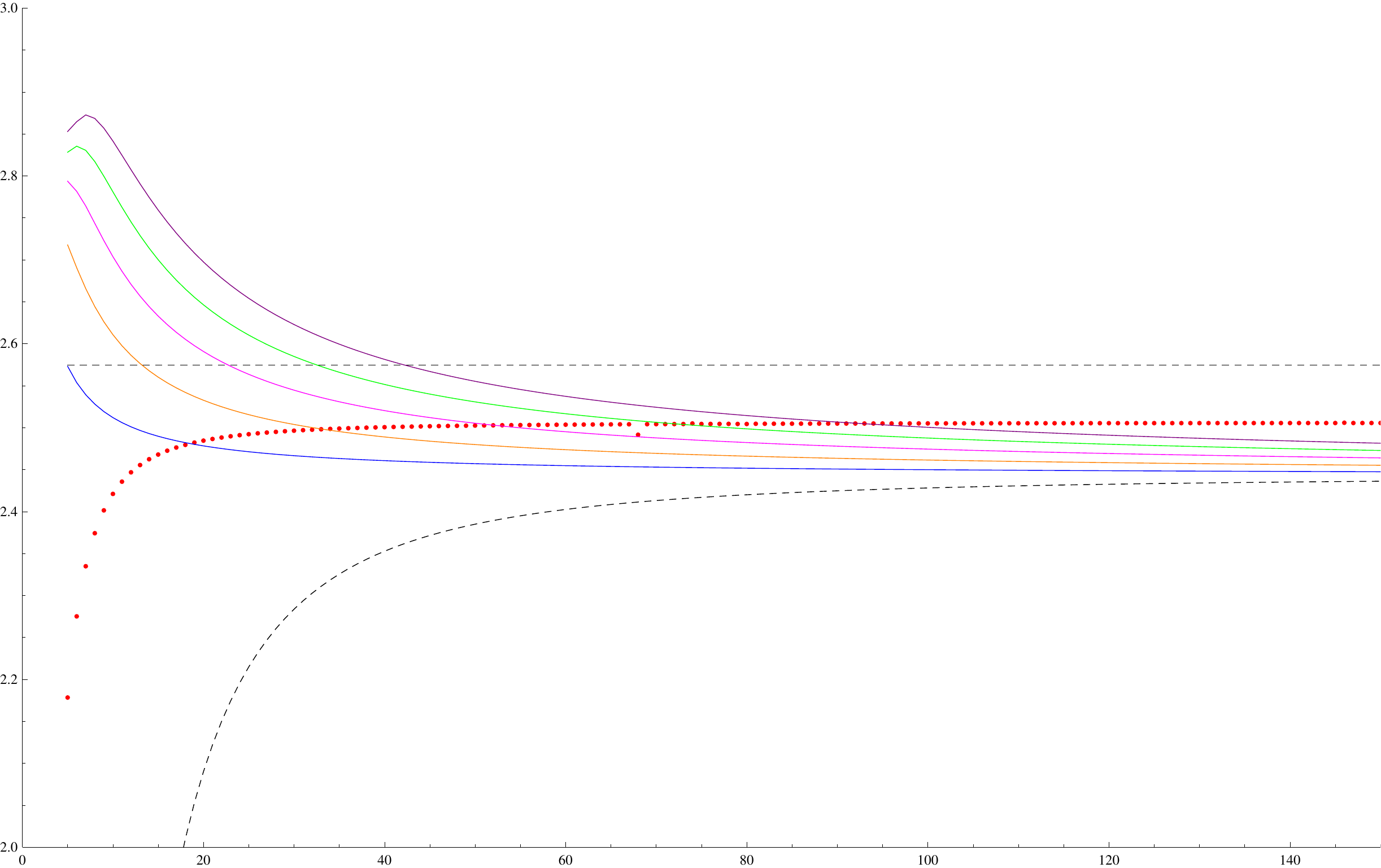}
\caption{Comparing $L_k^3(m)$ with the lower bound on the relative volume for $W(4,m)$}
\label{fig2}
\end{center}
\end{figure}

In Figure \ref{fig2}, the horizontal axis takes the values of $m$, the lower most curve and the top horizontal line represent the lower and upper bounds for the relative volume given in \eqref{eq19}. The dotted curve in red is for the values of relative volume computed using SnapPy \cite{cdgw}. The curves in blue, orange, magenta, green and purple represent the graphs of $L_k^3$ for $k=2,3,4,5$ and $6$ respectively. It is evident from the graphs that as $k$ value is higher, the lower bound for the relative volume gets better.

Similarly, for a fixed $n$, consider the following functions:
\begingroup
\allowdisplaybreaks
\begin{align*}
U_k^1(m)&= U_n\left(1-|1-f_k(m)|\right)\qquad\qquad\quad\;\,\,U_k^2(m)=\frac{U_n}{1+|1-f_k(m)|}\\[3pt]
U_k^3(m)&= U_n\left(1-|1-g_k(m)|\right)\qquad\text{and}\qquad U_k^4(m)=\frac{U_n}{1+|1-g_k(m)|}\;.
\end{align*}
\endgroup
\noindent Since from the experimental observations $f_k(m)$ and $g_k(m)$ converge to $1$ as $m$ grows large, the sequence $U_k^i(m)$, for $i=1,2,3$ and $4$, is increasing and $\lim_{m\to\infty}U_k^i(m)=U_n$. If the last inequality in \eqref{eq20} is strict, then one can observe that $v_n(m)\leq U_k^i(m)$ for $m$ large enough. This makes us to believe that each $U_k^i(m)$ provide better upper bound for the relative volume $v_n(m)$ than the upper bound as in \eqref{eq19}. As a sample for the family $W(4,m)$, Figure \ref{fig4} displays the graphs of $U_k^3$ for various values of $k$ along with the graphs of bounds on relative volume as in \eqref{eq19} and the values of relative volume computed using SnapPy \cite{cdgw}.

\begin{figure}[H]
\begin{center}
\includegraphics[width=1\linewidth]{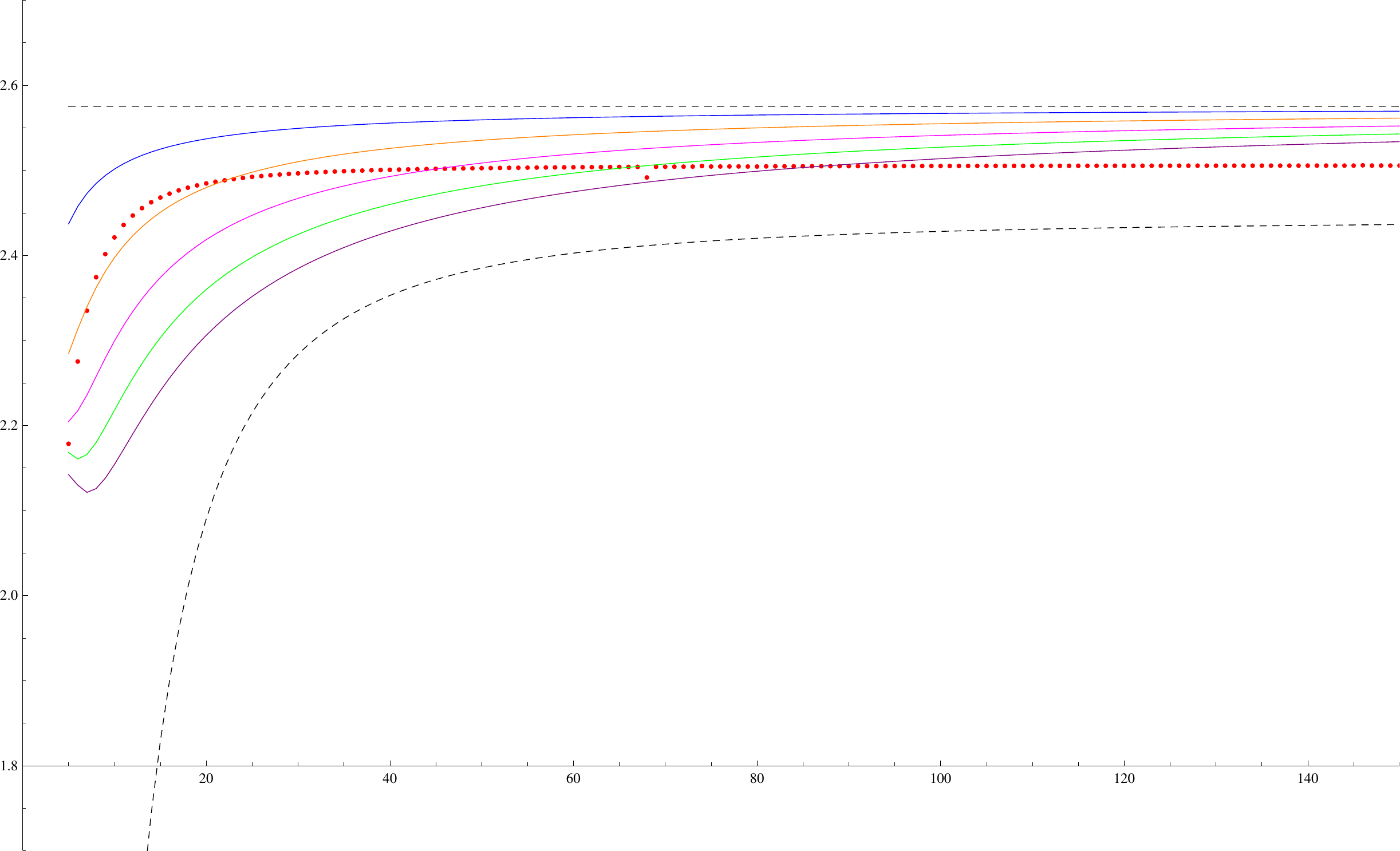}
\caption{Comparing $U_k^3(m)$ with the upper bound on the relative volume for $W(4,m)$}
\label{fig4}
\end{center}
\end{figure}

In Figure \ref{fig4}, the horizontal axis takes the values of $m$, the lower most curve and the top horizontal line represent the lower and upper bounds for the relative volume as in \eqref{eq19}. The dotted curve in red is for the values of relative volume computed using SnapPy \cite{cdgw}. The curves in blue, orange, magenta, green and purple represent the graphs of $U_k^3$ for $k=2,3,4,5$ and $6$ respectively. It is evident from the graphs that as $k$ value is higher, the upper bound for the relative volume gets better.

\begin{remark}
For fixed $n,i$ and $k$, if the first and the last inequalities in \eqref{eq20} are strict, then the functions $L_k^i(m)$ and $U_k^i(m)$ will certainly provide bounds for the relative volume of $W(n+1,m)$ for $m$ large enough. These bounds are better than the bounds in \eqref{eq19}. However, if the first inequality (respectively, the last inequality) in \eqref{eq20} is in fact an equality, then the function $L_k^i(m)$ for any $i$ and $k$ (respectively, the function $U_k^i(m)$ for any $i$ and $k$) may not remain a lower bound (respectively, an upper bound) for the relative volume.
\end{remark}

\section{Khovanov homology} \label{sec6}
 
A weaving knot $W(n+1,m)$ is an alternating knot or a nonsplit alternating link; therefore, its Khovanov homology \cite{k} is supported on two lines and can be determined by the Jones polynomial and the signature (see \cite{l1,l2,s1,s2}). Let us quote the following theorem.

\begin{theorem}[Theorem 1.2 and Theorem 4.5 in \cite{l1}]\label{thm2}
For an alternating knot (or a nonsplit alternating link) $L$, the Khovanov invariants $\mathcal{H}^{i,j}(L)$ are supported in two lines $j=2-\sigma(L)\pm1$, where $\sigma(L)$ is the signature of $L$.
\end{theorem} 

In \cite{ms}, the authors proved the following result regarding the signature of weaving knots.

\begin{proposition} 
For a weaving knot $W(n+1,m)$, the signature is $0$ if $n$ is even and it is $1-m$ if $n$ is odd.
\end{proposition}

Using this proposition, the direct application of Theorem \ref{thm2} will provide the following:

\begin{theorem}[Theorem 2.6 in \cite{ms}]\label{thm3}
For a weaving knot $W(n+1,m)$, the non-vanishing Khovanov homology $\mathcal{H}^{i,j}\!\left(W(n+1,m)\right)$ lies on the lines $j=2i\pm 1$ if $n$ is even and it lies on the lines $ j=2i+m-1\pm1$ if $n$ is odd.
\end{theorem}
 
For an alternating knot $L$, using the results in \cite[Section 1]{l1} along with the discussion in \cite[Section 6]{ms}, we can obtain the for the Khovanov polynomial $Kh(L)$ as follows:
\begin{equation}\label{eq21}
Kh(L)(t,Q)=Q^{-\sigma(L)}\left(Q^{-1}+Q\right)+\frac{Q^{-1}+tQ^3}{1+tQ^2}\left((-t)^{\frac{\sigma(L)}{2}}V_L\!\left(-tQ^2\right)-Q^{-\sigma(L)}\right)
\end{equation}
\noindent where $\sigma(L)$ and $V_L$ are the signature and the Jones polynomial of $L$ respectively. Note that the rational function $\frac{Q^{-1}+tQ^3}{1+tQ^2}$ evaluated at $t=t^{-1}$ and $Q=Q^{-1}$ is equal to itself. Let $L^\ast$ be the mirror image of $L$. Since $\sigma(L^\ast)=-\sigma(L)$ and $V_{L^\ast}(t)=V_L\!\left(t^{-1}\right)$, we have $Kh(L^\ast)(t,Q)=Kh(L)\!\left(t^{-1},Q^{-1}\right)$. Thus, $\text{rank}\;\mathcal{H}^{i,j}(L^\ast)=\text{rank}\;\mathcal{H}^{-i,-j}(L)$ for all $i$ and $j$. This suggest that the entries in the table of Khovanov ranks (the table of ranks of the Khovanov homology groups) for $L$ along the line $j=2i-\sigma(L)+1$ (respectively along the line $j=2i-\sigma(L)-1$) are same as the entries (read in the reverse direction) in the table of Khovanov ranks for $L^\ast$ along the line $j=2i-\sigma(L^\ast)-1$ (respectively along the line $j=2i-\sigma(L^\ast)+1$).

Let $m$ and $n$ be such that $\gcd(n+1,m)=1$. Then the weaving knot $W(n+1,m)$ is a link with one component, i.e. it is an alternating knot. Following the discussion just after \eqref{eq15}, we can obtain $V_{W(n+1,m)}(-tQ^2)$ by substituting $q=-tQ^2$, $z=\frac{-1}{1-tQ^2}$ and $w=\frac{-t^2Q^4}{1-tQ^2}$ in \eqref{eq11}. Using the obtained for $V_{W(n+1,m)}(-tQ^2)$, the equation \eqref{eq21}, in case of the weaving knot $W(n+1,m)$, takes the following form:
\begin{equation}\label{eq22}
Kh(W(n+1,m))(t,Q)=Q^{-\sigma}\left(Q^{-1}+Q\right)+\frac{Q^{-1}+tQ^3}{1+tQ^2}F(W(n+1,m))(t,Q)
\end{equation}
\noindent where
\begin{equation}\label{eq23}
F(W(n+1,m))(t,Q)=(-1)^{r+d}t^{-r}Q^{e-n}\left(1-tQ^2\right)^nTr(\rho(\sigma_{n+1,m}))-Q^{-\sigma}
\end{equation}
\noindent Here $d=\frac{1-(-1)^n}{2}$, $r=\frac{n-d}{2}$, $e=md$ is the exponent sum of $\sigma_{n+1,m}$ as a word in $\sigma_i$'s, $\sigma=(1-m)d$ is the signature of $W(n+1,m)$, and $Tr(\rho(\sigma_{n+1,m}))$ is evaluated at $q=-tQ^2$ and $z=\frac{-1}{1-tQ^2}$. Using \eqref{eq22}, we have the following Mathematica program to compute the Khovanov ranks (the ranks of the Khovanov homology groups) for weaving knots $W(n+1,m)$ with $\gcd(n+1,m)=1$. The program works along with the Mathematica program for the trace discussed in Section \ref{sec3}.\vskip5mm
\noindent\hspace{-3pt}\includegraphics[width=1.003\linewidth]{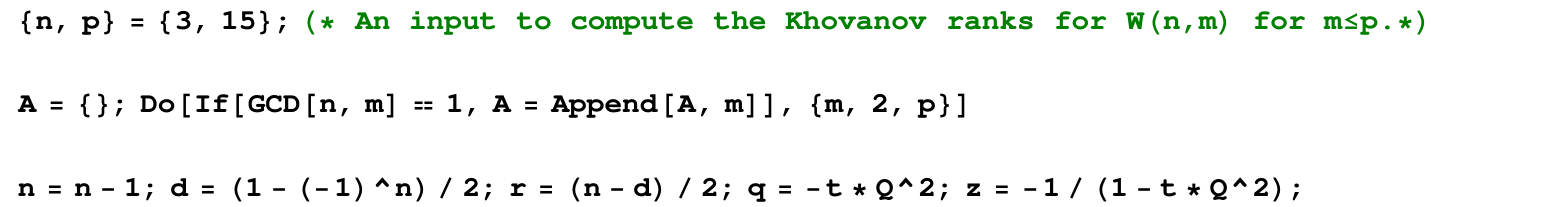}\vskip3mm
\noindent\hspace{-3pt}\includegraphics[width=1.003\linewidth]{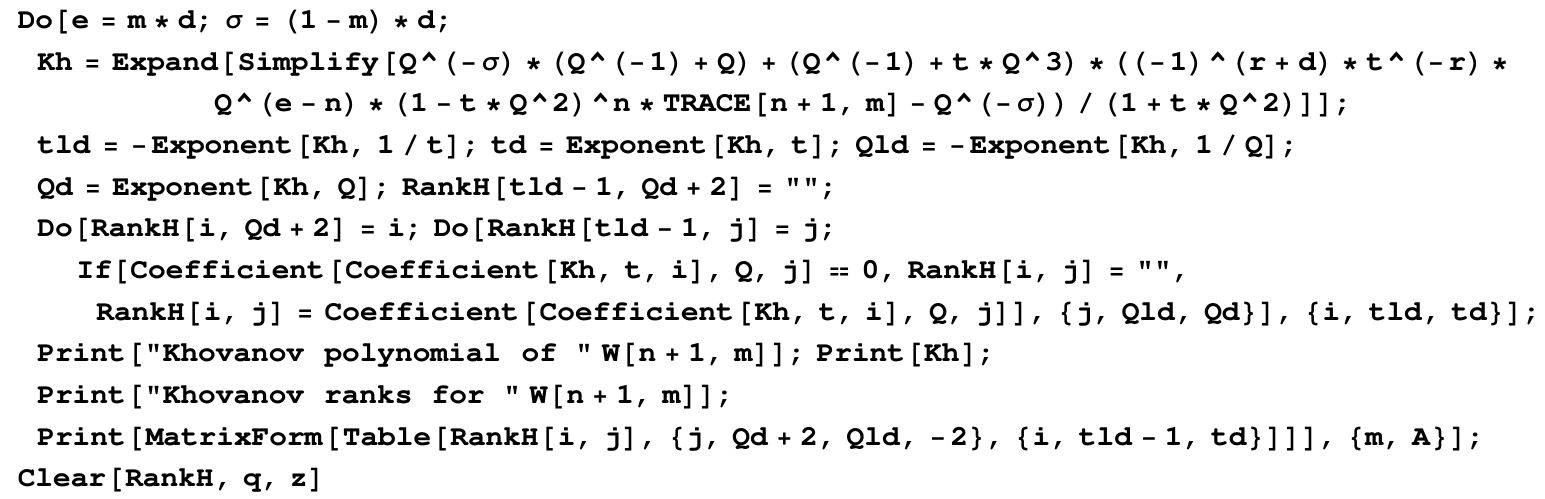}\vskip3mm
\noindent Below are tables of Khovanov ranks for some weaving knots computed using this program.

{\small
\begin{table}[H]
\begin{center}
\caption{Khovanov ranks for $W(4,3)$}
\(\begin{array}{rccccccccccc}\hline
& i & -3 & -2 & -1 & 0 & 1 & 2 & 3 & 4 & 5 & 6 \\
j &  &  &  &  &  &  &  &  &  &  & \\ \hline
15 &  &  &  &  &  &  &  &  &  &  & 1 \\
13 &  &  &  &  &  &  &  &  &  & 3 &  \\
11 &  &  &  &  &  &  &  &  & 5 & 1 &  \\
9 &  &  &  &  &  &  &  & 6 & 3 &  &  \\
7 &  &  &  &  &  &  & 7 & 5 &  &  &  \\
5 &  &  &  &  &  & 6 & 6 &  &  &  &  \\
3 &  &  &  &  & 5 & 7 &  &  &  &  &  \\
1 &  &  &  & 4 & 7 &  &  &  &  &  &  \\
-1 &  &  & 1 & 4 &  &  &  &  &  &  &  \\
-3 &  &  & 4 &  &  &  &  &  &  &  &  \\
-5 &  & 1 &  &  &  &  &  &  &  &  & \\ \hline
\end{array}\)\\
{\tiny Note: For the rank of $\mathcal{H}^{k,l}(W(4,3))$, look at the cell for which $i=k$ and $j=l$.}
\end{center}
\end{table}}

{\small
\begin{table}[H]
\begin{center}
\caption{Khovanov ranks for $W(5,3)$}
\(\begin{array}{rcccccccccccccc}\hline
& i & -6 & -5 & -4 & -3 & -2 & -1 & 0 & 1 & 2 & 3 & 4 & 5 & 6 \\
j &  &  &  &  &  &  &  &  &  &  &  &  &  & \\ \hline
13 &  &  &  &  &  &  &  &  &  &  &  &  &  & 1 \\
11 &  &  &  &  &  &  &  &  &  &  &  &  & 5 &  \\
9 &  &  &  &  &  &  &  &  &  &  &  & 11 & 1 &  \\
7 &  &  &  &  &  &  &  &  &  &  & 19 & 5 &  &  \\
5 &  &  &  &  &  &  &  &  &  & 25 & 11 &  &  &  \\
3 &  &  &  &  &  &  &  &  & 29 & 19 &  &  &  &  \\
1 &  &  &  &  &  &  &  & 30 & 25 &  &  &  &  &  \\
-1 &  &  &  &  &  &  & 25 & 30 &  &  &  &  &  &  \\
-3 &  &  &  &  &  & 19 & 29 &  &  &  &  &  &  &  \\
-5 &  &  &  &  & 11 & 25 &  &  &  &  &  &  &  &  \\
-7 &  &  &  & 5 & 19 &  &  &  &  &  &  &  &  &  \\
-9 &  &  & 1 & 11 &  &  &  &  &  &  &  &  &  &  \\
-11 &  &  & 5 &  &  &  &  &  &  &  &  &  &  &  \\
-13 &  & 1 &  &  &  &  &  &  &  &  &  &  &  & \\ \hline
\end{array}\)\\
{\tiny Note: For the rank of $\mathcal{H}^{k,l}(W(5,3))$, look at the cell for which $i=k$ and $j=l$.}
\end{center}
\end{table}}

In \cite[Section 7]{ms}, the authors have already pointed out that once we know the Khovanov ranks (the ranks of the Khovanov homology groups) for alternating knots, we can have the complete information of their integral Khovanov homology using the results of Lee \cite{l1,l2} and Shumakovitch \cite{s1,s2}. With this in mind, we can compute the integral Khovanov homology of a weaving knot $W(n+1,m)$ with $\gcd(n+1,m)=1$, since it is an alternating knot and we have \eqref{eq22} (in fact, we have a Mathematica program) to compute the Khovanov ranks for this type of a weaving knot.

For the rational Khovanov homology of a weaving knot $W(n+1,m)$ with $\gcd(n+1,m)=1$, we take advantage of the\;{\it \textquotedblleft knight move\textquotedblright\;rule} (see \cite[Section 7]{ms} and \cite[Section 1]{s2}) and simplify by recording the Khovanov ranks from only along the line $j=2i+(m-1)d+1$, where $d=\frac{1-(-1)^n}{2}$. In order to study the asymptotic behavior of the Khovanov homology, we have to normalize the data. This is done by computing the total rank (the sum of Khovanov ranks) along the line $j=2i+(m-1)d+1$ and dividing each rank by the total rank. This way, we obtain normalized ranks that sum to one. This raises the possibility of approximating the distribution of normalized ranks by a probability distribution. For our baseline experiments we choose to use the normal distribution $N(\mu,\sigma^2)$. The probability density function for the normal distribution is
\begin{equation}
f_{\mu,\sigma}(x)=\frac{1}{\sigma\sqrt{2\pi}}\exp\left(-\frac{\left(x-\mu\right)^2}{2\;\!\sigma^2}\right)
\end{equation}

We compute the values of the mean $\mu$ and the standard deviation $\sigma$ from our data. For a fixed $n$, we regard $X_m=W(n+1,m)$ as a countable family of random variables for $m\geq2$ with $\gcd(n+1,m)=1$. For each $m$, we do as follows: For a particular $i$, let $p_i$ denote the normalized rank of $\mathcal{H}^{i,j}(W(n+1,m))$ for $j=2i+(m-1)d+1$, where $d=\frac{1-(-1)^n}{2}$. From this data we find the mean $\mu_m$ using the formula $\mu_m=\Sigma_{i}\,i\,p_i$\, (see (2-13) in \cite{hmgb}). After knowing the value of $\mu_m$ we find the standard deviation $\sigma_m$ using the formula ${\sigma_m}\!^2=\Sigma_{i}\,i^2\,p_i-{\mu_m}\!^2$\, (see (2-16) in \cite{hmgb}). We use these values of $\mu_m$ and $\sigma_m$ and consider the density function
\begin{equation}
f_{\mu_m,\sigma_m}(x)=\frac{1}{\sigma_m\sqrt{2\pi}}\exp\left(-\frac{\left(x-\mu_m\right)^2}{2\;\!{\sigma_m}\!^2}\right)
\end{equation}

For few values of $n$ and $m$, we plot below (see Figure \ref{fig8} to Figure \ref{fig9}) the graph of the density function $f_{\mu_m,\sigma_m}$ together with the normalized Khovanov ranks $p_i$ for a weaving knot $W(n+1,m)$ along the line $j=2i+(m-1)d+1$. In each of the figure, the density function is plotted in blue curve and the normalized ranks are plotted in red dots.

\begin{figure}[H]
\begin{center}
\includegraphics[width=1\linewidth]{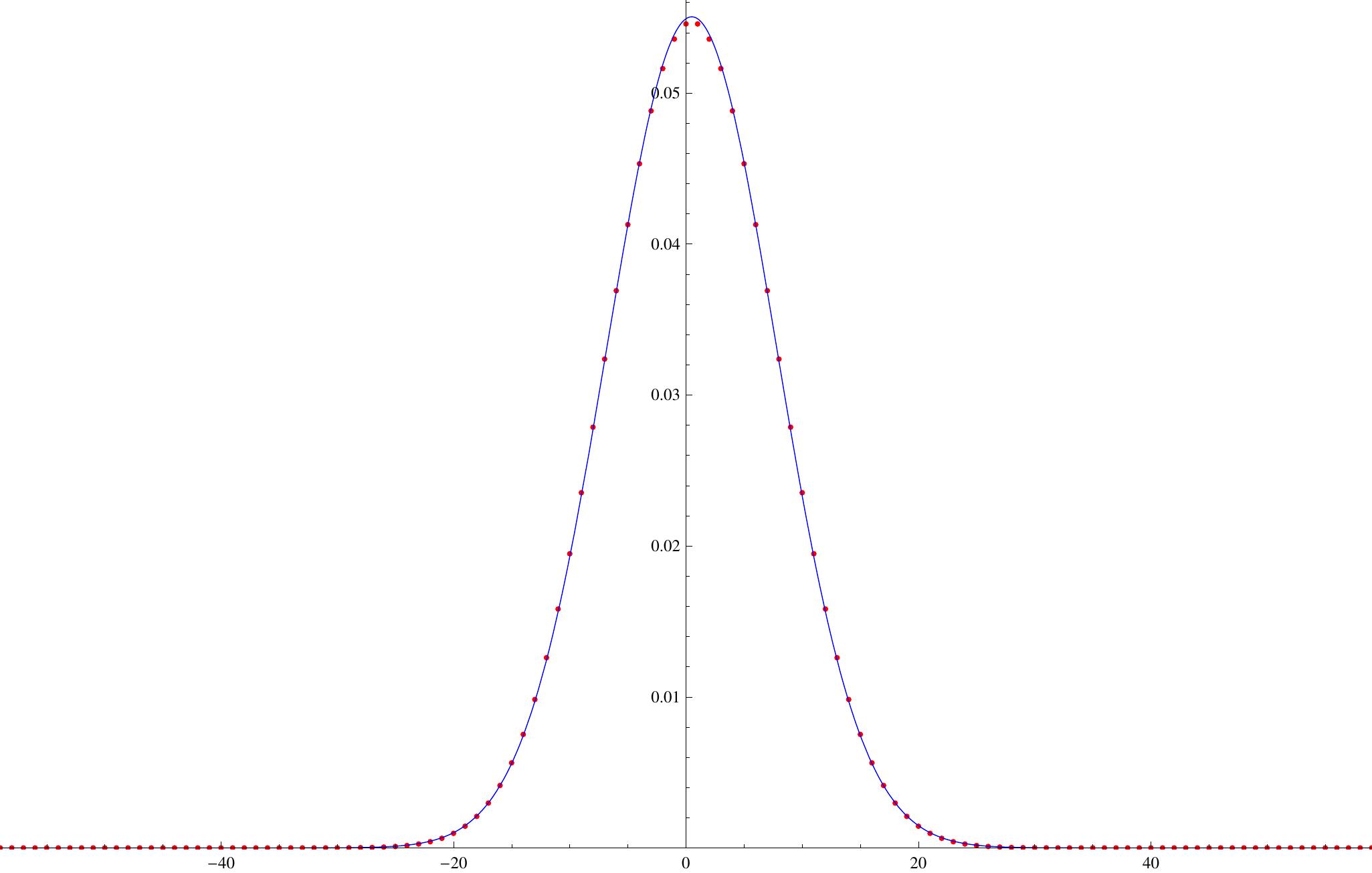}
\caption{The density function together with the normalized ranks for $W(3,59)$}
\label{fig8}
\end{center}
\end{figure}

\begin{figure}[H]
\begin{center}
\includegraphics[width=1\linewidth]{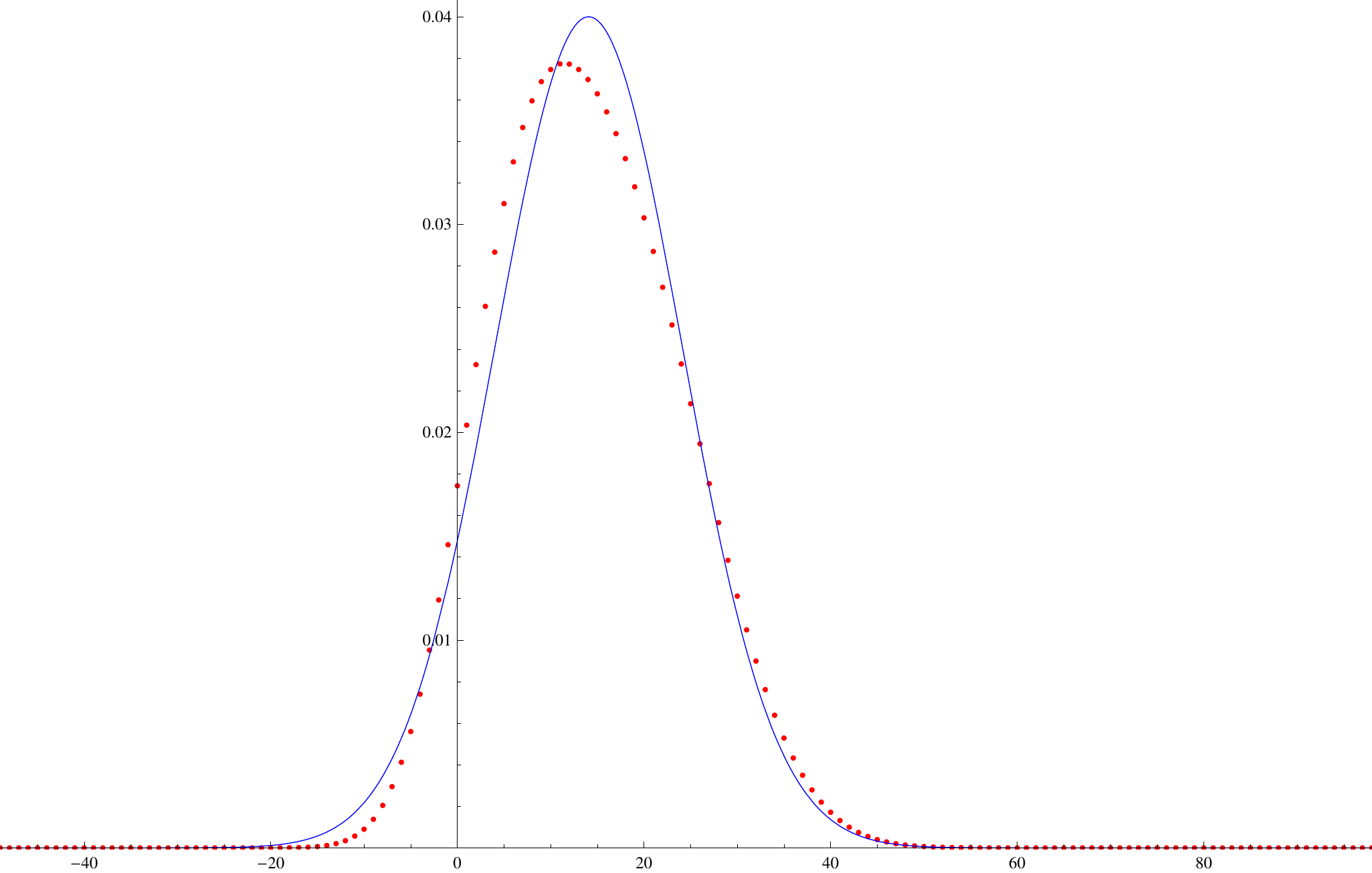}
\caption{The density function together with the normalized ranks for $W(4,49)$}
\end{center}
\end{figure}

\begin{figure}[H]
\begin{center}
\includegraphics[width=1\linewidth]{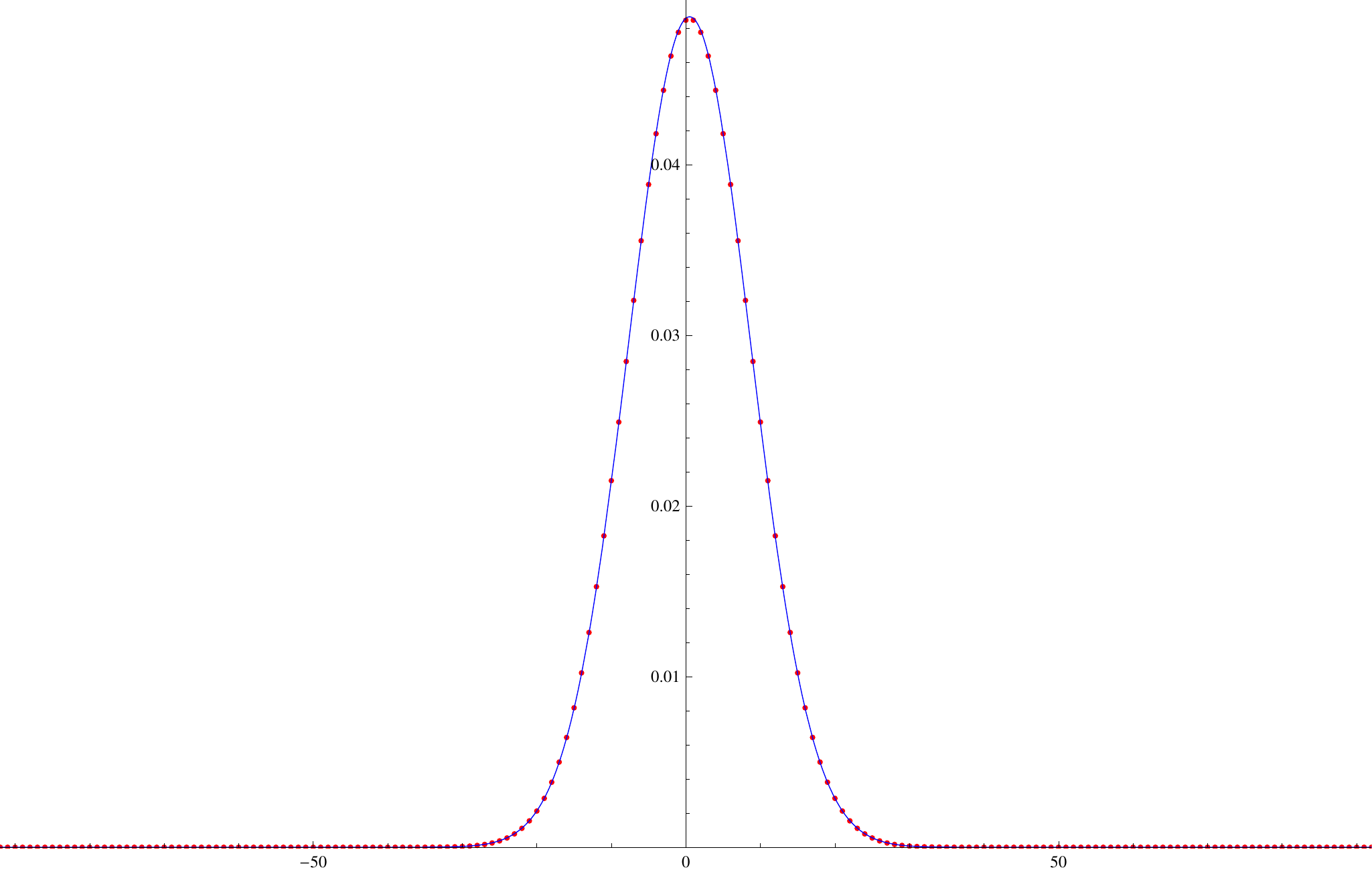}
\caption{The density function together with the normalized ranks for $W(5,46)$}
\end{center}
\end{figure}

\begin{figure}[H]
\begin{center}
\includegraphics[width=1\linewidth]{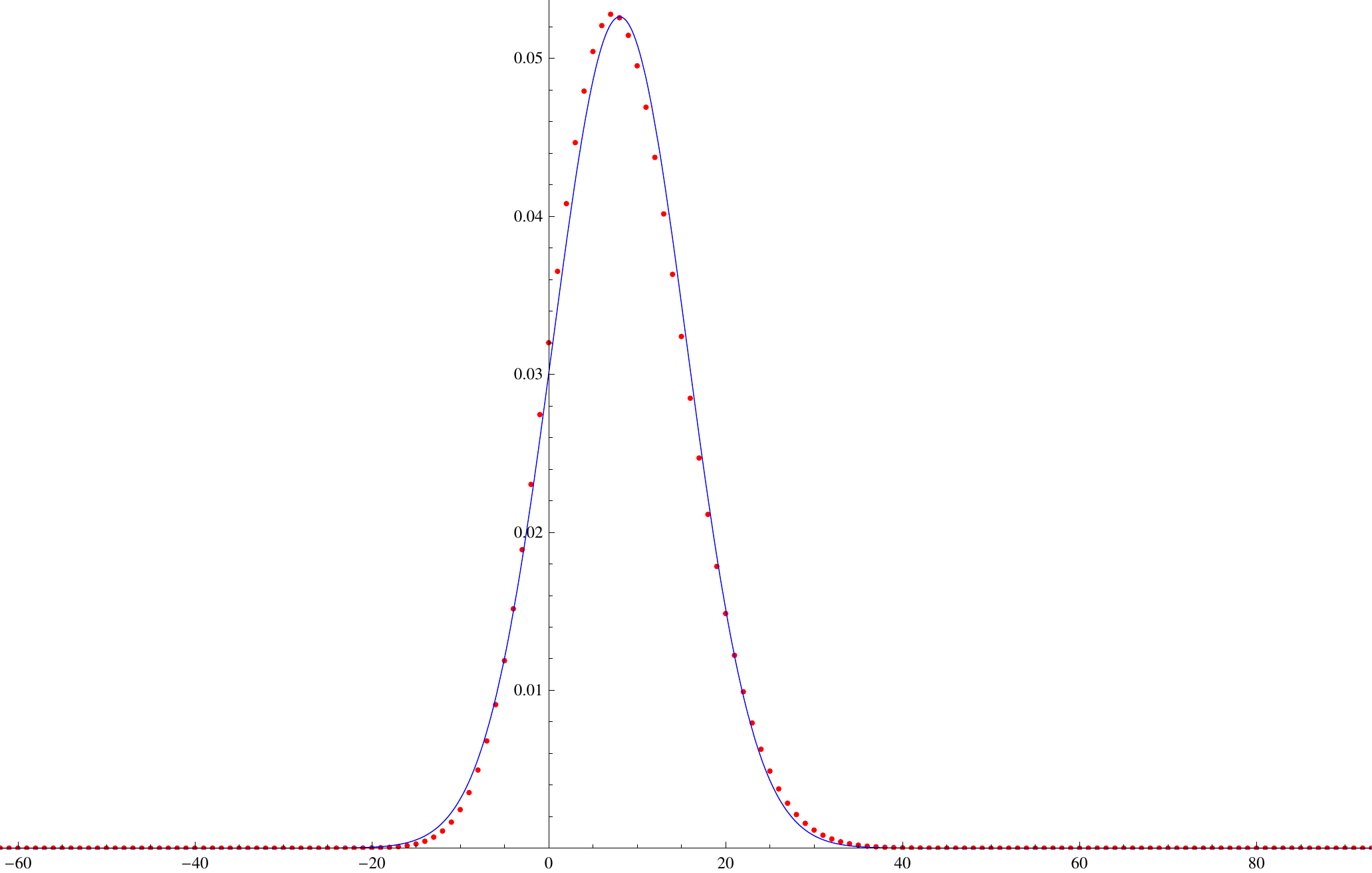}
\caption{The density function together with the normalized ranks for $W(6,31)$}
\label{fig9}
\end{center}
\end{figure}\vskip8mm

By looking at the figures above (Figure \ref{fig8} to Figure \ref{fig9}), it can be seen that, for a fixed $n$, the normalized Khovanov ranks for $W(n+1,m)$ approach to a normal distribution as $m$ grows large. For $n$ even this is very clear, but for $n$ odd, there seems to be some deviation of the normalized ranks from the density function. In the tables at the end (Table \ref{tb3} to Table \ref{tb4}), we have presented a data for the families $W(3,m)$, $W(4,m)$, $W(5,m)$ and $W(6,m)$ showing the $L^1$- and $L^2$-deviations (the $L^1$- and $L^2$-comparisons) of the normalized Khovanov ranks from the density function $f_{\mu_m,\sigma_m}$, where the $L^1$- and $L^2$-deviations are defined similarly as in \cite[Section 8]{ms}. The data provides sufficient evidence to conjecture the following.
\begin{conjecture}
For a fixed even integer $n$, the normalized Khovanov ranks for $W(n+1,m)$ along any of the two lines as in Theorem \ref{thm3} approach to a normal distribution as $m$ (being coprime to $n+1$) grows large.
\end{conjecture}

\subsection*{Acknowledgements}

The second author is thankful to Bhaskaracharya Pratishthana, Pune, India for the postdoctoral fellowship during this research work.

\begin{table}[H]
\caption{Values of $f_k(m)$ for $W(3,m)$}
\label{tb1}
\begin{center}
\(\begin{array}{rlllllll} \hline
\multicolumn{1}{c}{} & \multicolumn{1}{c}{k} & \multicolumn{1}{c}{2} & \multicolumn{1}{c}{3} & \multicolumn{1}{c}{4} & \multicolumn{1}{c}{5} & \multicolumn{1}{c}{6} & \multicolumn{1}{c}{7} \\
m &  &  &  &  &  &  & \\ \hline

35 &  & 0.971429 & 0.920816 & 0.855883 & 0.784009 & 0.711201 & 0.641677 \\
48 &  & 0.979167 & 0.940972 & 0.889703 & 0.829829 & 0.765494 & 0.70017 \\

100 &  & 0.99 & 0.9708 & 0.943446 & 0.909155 & 0.869239 & 0.825034 \\
113 &  & 0.99115 & 0.974078 & 0.949606 & 0.918705 & 0.882434 & 0.841892 \\
126 &  & 0.992063 & 0.976694 & 0.954559 & 0.926446 & 0.89323 & 0.85583 \\

178 &  & 0.994382 & 0.983399 & 0.967387 & 0.946757 & 0.921974 & 0.893543 \\
191 &  & 0.994764 & 0.984512 & 0.969538 & 0.950198 & 0.926901 & 0.900094 \\
256 &  & 0.996094 & 0.988403 & 0.977093 & 0.962366 & 0.944458 & 0.923629 \\

282 &  & 0.996454 & 0.989462 & 0.979161 & 0.965718 & 0.949329 & 0.930213 \\
347 &  & 0.997118 & 0.991421 & 0.982998 & 0.971962 & 0.958444 & 0.942592 \\
360 &  & 0.997222 & 0.991728 & 0.983602 & 0.972948 & 0.959887 & 0.94456 \\

412 &  & 0.997573 & 0.992766 & 0.985642 & 0.976283 & 0.964781 & 0.951245 \\
425 &  & 0.997647 & 0.992985 & 0.986075 & 0.976992 & 0.965824 & 0.952673 \\
438 &  & 0.997717 & 0.993192 & 0.986483 & 0.97766 & 0.966807 & 0.954018 \\

490 &  & 0.997959 & 0.993911 & 0.9879 & 0.979985 & 0.970231 & 0.958715 \\
503 &  & 0.998012 & 0.994067 & 0.988209 & 0.980492 & 0.970979 & 0.959744 \\
516 &  & 0.998062 & 0.994216 & 0.988503 & 0.980974 & 0.971691 & 0.960722 \\ \hline
\end{array}\)\\
{\tiny Note: For $1\leq k\leq8$, the values of $f_k(m)$ seem to approach $1$ as $m$ grows large.}
\end{center}
\end{table}

\begin{table}[H]
\caption{Values of $f_k(m)$ for $W(4,m)$}
\begin{center}
\(\begin{array}{rlllllll} \hline
\multicolumn{1}{c}{} & \multicolumn{1}{c}{k} & \multicolumn{1}{c}{2} & \multicolumn{1}{c}{3} & \multicolumn{1}{c}{4} & \multicolumn{1}{c}{5} & \multicolumn{1}{c}{6} & \multicolumn{1}{c}{7} \\
m &  &  &  &  &  &  & \\ \hline

15 &  & 0.961778 & 0.86242 & 0.722851 & 0.572935 & 0.434695 & 0.318976 \\
24 &  & 0.975694 & 0.90977 & 0.80967 & 0.691208 & 0.569055 & 0.453774 \\

78 &  & 0.992373 & 0.970744 & 0.935063 & 0.887946 & 0.832296 & 0.77064 \\
87 &  & 0.993156 & 0.97371 & 0.941511 & 0.898782 & 0.847997 & 0.791301 \\
105 &  & 0.994322 & 0.978141 & 0.951205 & 0.915196 & 0.871996 & 0.82322 \\

150 &  & 0.996018 & 0.984622 & 0.965504 & 0.939674 & 0.908264 & 0.872211 \\
168 &  & 0.996443 & 0.986253 & 0.969124 & 0.94592 & 0.917609 & 0.884979 \\
177 &  & 0.996623 & 0.986945 & 0.970663 & 0.948583 & 0.921603 & 0.890454 \\

231 &  & 0.99741 & 0.989974 & 0.977419 & 0.960309 & 0.939272 & 0.914799 \\
240 &  & 0.997507 & 0.990347 & 0.978254 & 0.961762 & 0.941471 & 0.917843 \\
249 &  & 0.997597 & 0.990694 & 0.979029 & 0.963113 & 0.943516 & 0.920677 \\

294 &  & 0.997964 & 0.99211 & 0.982201 & 0.968651 & 0.951919 & 0.932348 \\
303 &  & 0.998024 & 0.992343 & 0.982724 & 0.969565 & 0.953308 & 0.934282 \\
321 &  & 0.998135 & 0.99277 & 0.983682 & 0.971242 & 0.955859 & 0.937837 \\

366 &  & 0.998364 & 0.993655 & 0.98567 & 0.974723 & 0.961163 & 0.945242 \\
375 &  & 0.998403 & 0.993806 & 0.98601 & 0.975321 & 0.962075 & 0.946517 \\
384 &  & 0.99844 & 0.993951 & 0.986335 & 0.975891 & 0.962945 & 0.947733 \\ \hline
\end{array}\)\\
{\tiny Note: For $1\leq k\leq8$, the values of $f_k(m)$ seem to approach $1$ as $m$ grows large.}
\end{center}
\end{table}

\begin{table}[H]
\caption{Values of $f_k(m)$ for $W(5,m)$}
\begin{center}
\(\begin{array}{rllllllll} \hline
\multicolumn{1}{c}{} & \multicolumn{1}{c}{k} & \multicolumn{1}{c}{2} & \multicolumn{1}{c}{3} & \multicolumn{1}{c}{4} & \multicolumn{1}{c}{5} & \multicolumn{1}{c}{6} & \multicolumn{1}{c}{7} & \multicolumn{1}{c}{8} \\
m &  &  &  &  &  &  &  &\\ \hline

10 &  & 1 & 1.005 & 1.0185 & 1.04303 & 1.08045 & 1.13219 & 1.19931 \\
14 &  & 1 & 1.00255 & 1.00966 & 1.02291 & 1.04362 & 1.07289 & 1.11169 \\
22 &  & 1 & 1.00103 & 1.00399 & 1.00965 & 1.01868 & 1.03171 & 1.04928 \\
28 &  & 1 & 1.00064 & 1.00248 & 1.00604 & 1.01178 & 1.02012 & 1.03144 \\

32 &  & 1 & 1.00049 & 1.00191 & 1.00466 & 1.00911 & 1.0156 & 1.02445 \\
36 &  & 1 & 1.00039 & 1.00151 & 1.0037 & 1.00725 & 1.01245 & 1.01955 \\
48 &  & 1 & 1.00022 & 1.00085 & 1.0021 & 1.00414 & 1.00714 & 1.01127 \\
50 &  & 1 & 1.0002 & 1.00079 & 1.00194 & 1.00383 & 1.0066 & 1.01041 \\

56 &  & 1 & 1.00016 & 1.00063 & 1.00155 & 1.00306 & 1.00529 & 1.00837 \\
60 &  & 1 & 1.00014 & 1.00055 & 1.00135 & 1.00268 & 1.00463 & 1.00732 \\
68 &  & 1 & 1.00011 & 1.00043 & 1.00106 & 1.00209 & 1.00362 & 1.00574 \\
74 &  & 1 & 1.00009 & 1.00036 & 1.00089 & 1.00177 & 1.00307 & 1.00487 \\

80 &  & 1 & 1.00008 & 1.00031 & 1.00077 & 1.00152 & 1.00263 & 1.00418 \\
92 &  & 1 & 1.00006 & 1.00023 & 1.00058 & 1.00115 & 1.002 & 1.00318 \\
96 &  & 1 & 1.00005 & 1.00022 & 1.00053 & 1.00106 & 1.00184 & 1.00292 \\
104 &  & 1 & 1.00005 & 1.00018 & 1.00046 & 1.0009 & 1.00157 & 1.0025 \\
110 &  & 1 & 1.00004 & 1.00016 & 1.00041 & 1.00081 & 1.00141 & 1.00224 \\ \hline
\end{array}\)\\
{\tiny Note: For $1\leq k\leq8$, the values of $f_k(m)$ seem to approach $1$ as $m$ grows large.}
\end{center}
\end{table}

\begin{table}[H]
\caption{Values of $f_k(m)$ for $W(6,m)$}
\label{tb2}
\begin{center}
\(\begin{array}{rlllllll} \hline
\multicolumn{1}{c}{} & \multicolumn{1}{c}{k} & \multicolumn{1}{c}{2} & \multicolumn{1}{c}{3} & \multicolumn{1}{c}{4} & \multicolumn{1}{c}{5} & \multicolumn{1}{c}{6} & \multicolumn{1}{c}{7} \\
m &  &  &  &  &  &  & \\ \hline

5 &  & 0.990769 & 0.961371 & 0.920808 & 0.875241 & 0.825431 & 0.771134 \\
6 &  & 0.991453 & 0.963492 & 0.923062 & 0.876936 & 0.828974 & 0.779769 \\
7 &  & 0.992151 & 0.966097 & 0.926315 & 0.879296 & 0.829744 & 0.779752 \\

10 &  & 0.993846 & 0.973029 & 0.938623 & 0.895117 & 0.847111 & 0.79782 \\
12 &  & 0.994658 & 0.976488 & 0.945662 & 0.905692 & 0.860667 & 0.813796 \\
14 &  & 0.99529 & 0.979217 & 0.951466 & 0.914841 & 0.872928 & 0.828755 \\
17 &  & 0.996007 & 0.982338 & 0.958319 & 0.926024 & 0.888408 & 0.848154 \\

20 &  & 0.996538 & 0.984664 & 0.963551 & 0.934791 & 0.90086 & 0.864109 \\
22 &  & 0.996821 & 0.985907 & 0.966386 & 0.939616 & 0.907818 & 0.873149 \\
27 &  & 0.997362 & 0.98829 & 0.971887 & 0.949116 & 0.921716 & 0.891453 \\

31 &  & 0.997679 & 0.989689 & 0.975157 & 0.95484 & 0.930211 & 0.902796 \\
35 &  & 0.997928 & 0.990791 & 0.977751 & 0.959422 & 0.937072 & 0.912041 \\
36 &  & 0.997982 & 0.991031 & 0.978318 & 0.960428 & 0.938585 & 0.914089 \\

40 &  & 0.998173 & 0.991878 & 0.980325 & 0.964 & 0.943981 & 0.921423 \\
43 &  & 0.998294 & 0.992415 & 0.981603 & 0.966287 & 0.94745 & 0.926161 \\
47 &  & 0.998433 & 0.99303 & 0.98307 & 0.968921 & 0.951463 & 0.931664 \\
50 &  & 0.998523 & 0.99343 & 0.984027 & 0.970643 & 0.954096 & 0.935286 \\ \hline
\end{array}\)\\
{\tiny Note: For $1\leq k\leq8$, the values of $f_k(m)$ seem to approach $1$ as $m$ grows large.}
\end{center}
\end{table}

\begin{table}[H]
\caption{Deviations of the normalized ranks for $W(3,m)$ from the density function}
\label{tb3}
\begin{center}
\(\begin{array}{rrllll} \hline
m & \multicolumn{1}{c}{\text{Total rank}} & \multicolumn{1}{c}{\mu} & \multicolumn{1}{c}{\sigma} & \multicolumn{1}{c}{L^1\text{-deviation}} & \multicolumn{1}{c}{L^2\text{-deviation}} \\[2mm] \hline

47 & 2.20703\times 10^{19} & 0.5 & 6.46437 & 0.00747521 & 0.00141888 \\
61 & 1.56842\times 10^{25} & 0.5 & 7.36954 & 0.00567737 & 0.00101715 \\

89 & 7.92082\times 10^{36} & 0.5 & 8.90809 & 0.00387283 & 0.000629711 \\
103 & 5.6289\times 10^{42} & 0.5 & 9.5852 & 0.0033449 & 0.000523479 \\
131 & 2.8427\times 10^{54} & 0.5 & 10.813 & 0.00262246 & 0.000386449 \\

173 & 1.02022\times 10^{72} & 0.5 & 12.4292 & 0.00197886 & 0.000272269 \\
215 & 3.66147\times 10^{89} & 0.5 & 13.8583 & 0.00158861 & 0.000207168 \\
229 & 2.60201\times 10^{95} & 0.5 & 14.3029 & 0.00149012 & 0.000191383 \\

257 & 1.31407\times 10^{107} & 0.5 & 15.1531 & 0.00132765 & 0.000165575 \\
271 & 9.33837\times 10^{112} & 0.5 & 15.5608 & 0.00125766 & 0.000154909 \\
313 & 3.35145\times 10^{130} & 0.5 & 16.7244 & 0.00108851 & 0.000129288 \\

355 & 1.2028\times 10^{148} & 0.5 & 17.8121 & 0.000959583 & 0.000110402 \\
383 & 6.07439\times 10^{159} & 0.5 & 18.5018 & 0.000889239 & 0.000100378 \\
397 & 4.31674\times 10^{165} & 0.5 & 18.8371 & 0.000858056 & 0.0000959607 \\

425 & 2.18004\times 10^{177} & 0.5 & 19.4906 & 0.000800354 & 0.0000881039 \\
439 & 1.54924\times 10^{183} & 0.5 & 19.8092 & 0.000775335 & 0.0000845971 \\
446 & 1.30601\times 10^{186} & 0.5 & 19.9666 & 0.000763206 & 0.0000829365 \\ \hline
\end{array}\)\\
{\tiny Note: The values of both $L^1$ and $L^2\text{-deviations}$ seem to approach $0$ as $m$ grows large.}
\end{center}
\end{table}

\begin{table}[H]
\caption{Deviations of the normalized ranks for $W(4,m)$ from the density function}
\begin{center}
\(\begin{array}{rrllll} \hline
m & \multicolumn{1}{c}{\text{Total rank}} & \multicolumn{1}{c}{\mu} & \multicolumn{1}{c}{\sigma} & \multicolumn{1}{c}{L^1\text{-deviation}} & \multicolumn{1}{c}{L^2\text{-deviation}} \\[2mm] \hline

13 & 8.85526\times 10^7 & 4.96218 & 4.10061 & 0.0632557 & 0.01523 \\
31 & 4.16558\times 10^{18} & 9.55151 & 7.16703 & 0.0896461 & 0.0170228 \\
49 & 1.29887\times 10^{29} & 14.0752 & 9.97569 & 0.101445 & 0.016681 \\

67 & 3.50349\times 10^{39} & 18.5862 & 12.6958 & 0.112183 & 0.0166565 \\
85 & 8.768\times 10^{49} & 23.0925 & 15.374 & 0.124524 & 0.0168477 \\
91 & 2.53634\times 10^{53} & 24.594 & 16.2609 & 0.128625 & 0.0169273 \\
109 & 5.99307\times 10^{63} & 29.0977 & 18.9103 & 0.140952 & 0.0171587 \\

133 & 3.89776\times 10^{77} & 35.101 & 22.4241 & 0.156364 & 0.0173918 \\
157 & 2.45248\times 10^{91} & 41.1033 & 25.9246 & 0.170129 & 0.0175175 \\
175 & 5.39261\times 10^{101} & 45.6046 & 28.5442 & 0.179508 & 0.0175495 \\

187 & 4.207\times 10^{108} & 48.6053 & 30.2885 & 0.185228 & 0.0175464 \\
205 & 9.0979\times 10^{118} & 53.1063 & 32.9026 & 0.193235 & 0.017512 \\
235 & 1.50204\times 10^{136} & 60.6075 & 37.2544 & 0.205074 & 0.0173945 \\

241 & 4.16213\times 10^{139} & 62.1077 & 38.1241 & 0.207211 & 0.0173642 \\
259 & 8.82378\times 10^{149} & 66.6083 & 40.7325 & 0.213363 & 0.0172627 \\
277 & 1.86162\times 10^{160} & 71.1088 & 43.3396 & 0.218994 & 0.0171489 \\
295 & 3.91101\times 10^{170} & 75.6092 & 45.9457 & 0.224201 & 0.0170261 \\ \hline
\end{array}\)
\end{center}
\end{table}

\begin{table}[H]
\caption{Deviations of the normalized ranks for $W(5,m)$ from the density function}
\begin{center}
\(\begin{array}{rrllll} \hline
m & \multicolumn{1}{c}{\text{Total rank}} & \multicolumn{1}{c}{\mu} & \multicolumn{1}{c}{\sigma} & \multicolumn{1}{c}{L^1\text{-deviation}} & \multicolumn{1}{c}{L^2\text{-deviation}} \\[2mm] \hline

9 & 7.18681\times 10^7 & 0.5 & 3.60626 & 0.0125854 & 0.00324384 \\
11 & 4.70654\times 10^9 & 0.5 & 3.98674 & 0.00953972 & 0.0022939 \\
17 & 1.30183\times 10^{15} & 0.5 & 4.96621 & 0.00576828 & 0.00125091 \\

21 & 5.518\times 10^{18} & 0.5 & 5.52486 & 0.00467484 & 0.000957202 \\
33 & 4.20138\times 10^{29} & 0.5 & 6.93607 & 0.00295785 & 0.000542858 \\
39 & 1.1593\times 10^{35} & 0.5 & 7.54332 & 0.00249014 & 0.000440293 \\

43 & 4.91358\times 10^{38} & 0.5 & 7.92233 & 0.00226405 & 0.000389583 \\
47 & 2.08258\times 10^{42} & 0.5 & 8.28402 & 0.00207088 & 0.000348495 \\
51 & 8.8268\times 10^{45} & 0.5 & 8.63057 & 0.00190276 & 0.000314598 \\

61 & 1.03231\times 10^{55} & 0.5 & 9.44145 & 0.00159543 & 0.000251396 \\
69 & 1.85445\times 10^{62} & 0.5 & 10.0431 & 0.00140353 & 0.000215446 \\
77 & 3.33135\times 10^{69} & 0.5 & 10.6107 & 0.00126212 & 0.000187797 \\

83 & 9.1923\times 10^{74} & 0.5 & 11.0173 & 0.00117013 & 0.00017096 \\
89 & 2.53646\times 10^{80} & 0.5 & 11.4093 & 0.00108826 & 0.000156657 \\
97 & 4.55652\times 10^{87} & 0.5 & 11.912 & 0.00100104 & 0.000140657 \\

103 & 1.2573\times 10^{93} & 0.5 & 12.2755 & 0.000942445 & 0.000130478 \\
109 & 3.4693\times 10^{98} & 0.5 & 12.6286 & 0.000888943 & 0.000121553 \\ \hline
\end{array}\)\\
{\tiny Note: The values of both $L^1$ and $L^2\text{-deviations}$ seem to approach $0$ as $m$ grows large.}
\end{center}
\end{table}

\begin{table}[H]
\caption{Deviations of the normalized ranks for $W(6,m)$ from the density function}
\label{tb4}
\begin{center}
\(\begin{array}{rrllll} \hline
m & \multicolumn{1}{c}{\text{Total rank}} & \multicolumn{1}{c}{\mu} & \multicolumn{1}{c}{\sigma} & \multicolumn{1}{c}{L^1\text{-deviation}} & \multicolumn{1}{c}{L^2\text{-deviation}} \\[2mm] \hline

5 & \multicolumn{1}{c}{254403} & 2.66802 & 2.99802 & 0.0264457 & 0.00759375 \\
7 & 5.68714\times 10^7 & 3.2106 & 3.47257 & 0.0220524 & 0.00579802 \\
11 & 2.2178\times 10^{12} & 4.11453 & 4.32385 & 0.0303306 & 0.00702358 \\
13 & 4.12427\times 10^{14} & 4.53041 & 4.71055 & 0.0337645 & 0.00747458 \\
17 & 1.33526\times 10^{19} & 5.33251 & 5.4275 & 0.0385126 & 0.00800467 \\
19 & 2.34815\times 10^{21} & 5.72499 & 5.76377 & 0.0403142 & 0.00816211 \\
23 & 7.03737\times 10^{25} & 6.50024 & 6.40244 & 0.043491 & 0.00836267 \\
25 & 1.20359\times 10^{28} & 6.88457 & 6.70782 & 0.0446933 & 0.00842603 \\
29 & 3.45657\times 10^{32} & 7.64892 & 7.29618 & 0.0470285 & 0.00850699 \\
31 & 5.81386\times 10^{34} & 8.02949 & 7.58081 & 0.0479072 & 0.00853115 \\
35 & 1.6251\times 10^{39} & 8.78834 & 8.13422 & 0.0498128 & 0.0085572 \\
37 & 2.70315\times 10^{41} & 9.16687 & 8.404 & 0.0504798 & 0.00856184 \\
41 & 7.41587\times 10^{45} & 9.92257 & 8.93174 & 0.0519877 & 0.00855916 \\
43 & 1.22377\times 10^{48} & 10.2999 & 9.19037 & 0.052611 & 0.00855324 \\
47 & 3.31162\times 10^{52} & 11.0536 & 9.6985 & 0.0538845 & 0.0085348 \\
49 & 5.43243\times 10^{54} & 11.4301 & 9.94846 & 0.0544588 & 0.00852309 \\ \hline
\end{array}\)
\end{center}
\end{table}


\normalsize

\begin{thebibliography}{[HD82]}
	
\normalsize
\baselineskip=17pt

\bibitem{a} J. W. Alexander, {\it Topological invariants of knots and links}, Trans. Amer. Math. Soc. 30 (1928), 275--306.
	
\bibitem{b} S. J. Bigelow, {\it Braid groups and Iwahori-Hecke algebras}, in: Problems on mapping class groups and related topics, Proc. Sympos. Pure Math. 74, Amer. Math. Soc., Providence, RI, 2006, 285–299.

\bibitem{bz} G. Burde and H. Zieschang, {\it Knots}, 2nd ed., de Gruyter Studies in Mathematics 5, Walter de Gruyter, Berlin, 2003.

\bibitem{ckp1} A. Champanerkar, I. Kofman and J. S. Purcell, {\it Geometrically and diagrammatically maximal knots}, J. Lond. Math. Soc. 94 (2016) 883--908.

\bibitem{ckp2} A. Champanerkar, I. Kofman and J. S. Purcell, {\it Volume bounds for weaving knots}, Algebr. Geom. Topol. 16 (2016) 3301--3323.

\bibitem{cdgw} M. Culler, N. M. Dunfield, M. Goerner, and J. R. Weeks, {\it SnapPy, a computer program for studying the geometry and topology of $3$-manifolds}, http://snappy.computop.org.

\bibitem{dl} O. T. Dasbach and X.-S. Lin, {\it A volumish theorem for the Jones polynomial of alternating knots}, Pacific J. Math. 231 (2007) 279--291.

\bibitem{homfly} P. Freyd, D. Yetter, J. Hoste, W. B. R. Lickorish, K. Millett and A. Ocneanu, {\it A new polynomial invariant of knots and links}, Bull. Amer. Math. Soc. (N. S.) 12 (1985) 239--246.

\bibitem{hkw} P. D. L. Harpe, M. Kervaire and C. Weber, {\it On the Jones polynomial}, Enseign. Math. 32 (1986) 271--335.

\bibitem{hmgb} W. W. Hines, D. C. Montgometry, D. M. Goldsman and C. M. Borror, {\it Probability and statistics in engineering}, 4th ed., John Wiley \& Sons, 2003.

\bibitem{j1} V. F. R. Jones, {\it A polynomial invariant for knots via von Neumann algebras}, Bull. Amer. Math. Soc. (N. S.) 12 (1985), 103--111.

\bibitem{j2} V. F. R. Jones, {\it Hecke algebra representations of braid groups and link polynomials}, Ann. of Math. 126 (1987) 335--388.

\bibitem{k} M. Khovanov, {\it A categorification of the Jones polynomial}, Duke Math. J. 101 (2000) 359--426.

\bibitem{l1} E. S. Lee, {\it An endomorphism of the Khovanov invariant}, Adv. Math. 197 (2005) 554--586.

\bibitem{l2} E. S. Lee, {\it The support of the Khovanov's invariants for alternating knots}, arXiv:math/0201105v1 (2002).

\bibitem{lm} W. B. R. Lickorish and K. C. Millett, {\it A polynomial invariant of oriented links}, Topology 26 (1987) 107--141.

\bibitem{mr} R. Mishra and H. Raundal, {\it Hecke algebra trace algorithm and some conjectures on weaving knots}, Notebook Archive, https://notebookarchive.org/2020-01-7fugqlj, 2020.

\bibitem{ms} R. Mishra and R. Staffeldt, {\it Polynomial invariants, knot homologies, and higher twist numbers of weaving knots $W(3,n)$}, arXiv:1902.01819v2 (2019).

\bibitem{m} H. Murakami, {\it An introduction to the volume conjecture and its generalizations}, Acta Math. Vietnam. 33 (2008) 219--253.

\bibitem{pt} J. H. Przytycki and P. Traczyk, {\it Conway algebras and skein equivalence of links}, Proc. Amer. Math. Soc. 100 (1987), 744--748.
 
\bibitem{s1} A. N. Shumakovitch, {\it Torsion in Khovanov homology of homologically thin knots}, arXiv:1806.05168v1 (2018).

\bibitem{s2} A. N. Shumakovitch, {\it Torsion of Khovanov homology}, Fund. Math. 225 (2014) 343--364.

\bibitem{t} W. P. Thurston, {\it Three dimensional manifolds, Kleinian groups and hyperbolic geometry}, Bull. Amer. Math. Soc. (N. S.) 6 (1982) 357--381.

\end{thebibliography}
\end{document}